\RequirePackage{iftex}
\ifluatex
\pdfvariable suppressoptionalinfo \numexpr1+2+4+8+32+64+512\relax
\fi
\ifxetex
    \makeatletter
    \newcommand{\overlaynot}[2]{\sbox\z@{$\m@th#1\notaccent{}$}\sbox\tw@{$\m@th#1#2$}\dimen@=\dimexpr(\ht\tw@-\ht\z@)/2\relax
        \vphantom{\raisebox{\dimen@}{\copy\z@}}\ooalign{\hidewidth\raisebox{\dimen@}{\box\z@}\hidewidth\cr\box\tw@}
    }
    \makeatother
\fi

\RequirePackage[l2tabu, orthodox]{nag}

\documentclass[a4paper,11pt]{article}
\usepackage{etoolbox}
\newtoggle{book}

\usepackage{geometry}

\usepackage[hidelinks, breaklinks, pdfnewwindow=true, unicode=true]{hyperref}

\usepackage{mleftright}
\usepackage{mathtools}
\usepackage{nccmath}
\mathtoolsset{mathic}
\usepackage{amsthm}
\iftutex
    \usepackage{fontspec}
    \usepackage{unicode-math}
\else
    \usepackage{amssymb}
    \let\symbb\mathbb
    \let\symcal\mathcal
    
    \let\symbf\mathbf
    \let\symfrak\mathfrak
\DeclareSymbolFont{stmry}{U}{stmry}{m}{n}
    \DeclareMathDelimiter\lBrack{\mathopen}{stmry}{"4A}{stmry}{"71}
    \DeclareMathDelimiter\rBrack{\mathclose}{stmry}{"4B}{stmry}{"79}
    \DeclareSymbolFont{wasy}{U}{wasy}{m}{n}
    \DeclareMathSymbol\mdwhtsquare{\mathord}{wasy}{"32}
    \DeclareMathSymbol\mdwhtdiamondold{\mathord}{wasy}{"33}
    \newcommand{\mdwhtdiamond}{\mathchoice{\raisebox{-.09em}{$\displaystyle\mdwhtdiamondold$}}
             {\raisebox{-.09em}{$\mdwhtdiamondold$}}
             {\raisebox{-0.05em}{$\scriptstyle\mdwhtdiamondold$}}
             {\raisebox{-0.02em}{$\scriptscriptstyle\mdwhtdiamondold$}}}
    \renewcommand{\boxdot}{\mathord{\Box\kern-1.2ex\cdot\kern.55ex}}
\fi
\usepackage{microtype}

\usepackage[british]{babel}
\usepackage[autostyle]{csquotes}
\usepackage{authblk} \usepackage{orcidlink} 

\ExplSyntaxOn
\msg_redirect_name:nnn { kernel } { variant-same-as-base } { info }
\ExplSyntaxOff

\usepackage{semantex}

\ExplSyntaxOn
\msg_redirect_name:nnn { kernel } { variant-same-as-base } { error }
\ExplSyntaxOff

\usepackage{enumitem}
\usepackage{tasks}
\DeclareInstance {tasks} {thmenumerate} {default} { label = (\roman*), label-width = 2em }

\usepackage{caption}
\usepackage{booktabs}
\usepackage{tabularray}
\UseTblrLibrary{amsmath}
\UseTblrLibrary{booktabs}

\usepackage[sortcites,giveninits,isbn=false,sorting=nyt,date=year]{biblatex}
\renewbibmacro*{in:}{\ifentrytype{article}{}{\bibstring{in}\printunit{\intitlepunct}}}

\renewbibmacro*{title}{\ifboolexpr{
    test {\iffieldundef{title}}
    and
    test {\iffieldundef{subtitle}}
  }
    {}
    {\printtext[title]{\printfield[title-case]{title}\setunit{\subtitlepunct}\printfield[subtitle-case]{subtitle}}\newunit}\printfield{titleaddon}}
\DeclareFieldFormat{title-case}{#1}
\DeclareFieldFormat[article,inbook,incollection,inproceedings,inreference]{title-case}{\MakeSentenceCase*{#1}}
\DeclareFieldFormat{subtitle-case}{#1}
\DeclareFieldFormat[article,inbook,incollection,inproceedings,inreference]{subtitle-case}{\MakeSentenceCase*{#1}}

\usepackage{fnpct}

\usepackage{zref}
\usepackage{zref-clever}

\usepackage{tikz}
\usetikzlibrary{shapes.geometric}
\usetikzlibrary{arrows}
\usetikzlibrary{fit}

\tikzset{global scale/.style={
    scale=#1,
    every node/.style={scale=#1}
  }
}
\makeatletter
\pgfdeclareshape{heart}{
  \inheritsavedanchors[from=circle]
  \inheritanchorborder[from=circle]
  \inheritanchor[from=circle]{center}

  \backgroundpath{
    \pgfpathmoveto{\pgfpoint{0pt}{3.3pt}}
    \pgfpathcurveto
      {\pgfpoint{-12.1pt}{15.4pt}}
      {\pgfpoint{-19.8pt}{-1.1pt}}
      {\pgfpoint{0pt}{-12.1pt}}
    \pgfpathcurveto
      {\pgfpoint{19.8pt}{-1.1pt}}
      {\pgfpoint{12.1pt}{15.4pt}}
      {\pgfpoint{0pt}{3.3pt}}
    \pgfpathclose
  }
}
\makeatother

\NewDocumentCommand{\zsubref}{m o m}{\IfValueTF{#2}{\zcref{#1}[#2]}{\zcref{#1}}~\zref{#3}}
\NewDocumentCommand{\citepage}{m}{p.~#1}
\newcommand\sentencespace{\spacefactor=3000{}\space\ignorespaces}

\newcommand{\tdefiff}{if} 

\setlist[description]{labelindent=0.5\leftmargin}

\NewDocumentCommand{\tCIP}{}{C\kern0.034em I\kern0.025em P}
\NewDocumentCommand{\titCIP}{}{C\kern0.014em I\kern0.009em P}
\NewDocumentCommand{\tDIP}{}{D\kern0.03em I\kern0.025em P}
\NewDocumentCommand{\tpmorph}{}{p\babelhyphen{nobreak}morphism}
\NewDocumentCommand{\tKaxiom}{}{K\babelhyphen{nobreak}axiom}
\NewDocumentCommand{\tadequate}{m}{\(#1\)\babelhyphen{nobreak}adequate}
\NewDocumentCommand{\tframe}{m}{\(#1\)\babelhyphen{nobreak}frame}
\NewDocumentCommand{\tframes}{m}{\(#1\)\babelhyphen{nobreak}frames}

\iftutex

[
    Extension = {.otf},
    UprightFont={*-Regular},
    UprightFeatures = {
        SizeFeatures = {
            {Size={-9}, Font={NewCM08-Regular}},
            {Size={9-}},
        },
    },
    BoldFont={*-Bold},
    ItalicFont={*-Italic},
    ItalicFeatures = {
        SizeFeatures = {
            {Size={-9}, Font={NewCM08-Italic}},
            {Size={9-}},
        },
    },
    BoldItalicFont={*-BoldItalic},
SmallCapsFeatures={Numbers=OldStyle},
SlantedFont={*-Regular},
    SlantedFeatures={FakeSlant=0.25},
    BoldSlantedFont={*-Bold},
    BoldSlantedFeatures={FakeSlant=0.25},
ItalicFeatures = { SmallCapsFont={NewCM10-Regular.otf}, SmallCapsFeatures={Letters=SmallCaps, FakeSlant=0.25}},
    BoldItalicFeatures = { SmallCapsFont={NewCM10-Bold.otf}, SmallCapsFeatures={Letters=SmallCaps, FakeSlant=0.25}},
]

[
    ItalicFont=NewCMSans10-Oblique.otf,
    BoldFont=NewCMSans10-Bold.otf,
    BoldItalicFont=NewCMSans10-BoldOblique.otf,
    SlantedFont=NewCMSans10-Oblique.otf,
    BoldSlantedFont=NewCMSans10-BoldOblique.otf,
    SmallCapsFeatures={Numbers=OldStyle},
]

[
    ItalicFont=NewCMMono10-Italic.otf,
    BoldFont=NewCMMono10-Bold.otf,
    BoldItalicFont=NewCMMono10-BoldOblique.otf,
    SlantedFont=NewCMMono10-Regular.otf,
    SlantedFeatures={FakeSlant=0.25},
    BoldSlantedFont=NewCMMono10-Bold.otf,
    BoldSlantedFeatures={FakeSlant=0.25},
    SmallCapsFeatures={Numbers=OldStyle},
]

 \fi
\ifluatex
    \directlua{require("modules.hyphenate")}
 \fi

\zcRefTypeSetup{equation}{abbrev}
\zcLanguageSetup{english}{
  cap = true ,
  type = page ,
    cap = false ,
  type = equation ,
    cap = false ,
}
\zcLanguageSetup{dutch}{
  cap = true ,
  type = page ,
    cap = false ,
  type = equation ,
    cap = false ,
}
\zcLanguageSetup{english}{
  type = commentary ,
    Name-sg = Commentary ,
    name-sg = commentary ,
    Name-pl = Commentaries ,
    name-pl = commentaries ,
  type = claim ,
    Name-sg = Claim ,
    name-sg = claim ,
    Name-pl = Claims ,
    name-pl = claims ,
  type = axiom ,
    Name-sg = Axiom ,
    name-sg = axiom ,
    Name-pl = Axiomata ,
    name-pl = axiomata ,
  type = notation ,
    Name-sg = Notation ,
    name-sg = notation ,
    Name-pl = Notations ,
    name-pl = notations ,
  type = convention ,
    Name-sg = Convention ,
    name-sg = convention ,
    Name-pl = Conventions ,
    name-pl = conventions ,
  type = terminology,
    Name-sg = Terminology ,
    name-sg = terminology ,
    Name-pl = Terminologies ,
    name-pl = terminologies ,
  type = problem ,
    Name-sg = Problem ,
    name-sg = problem ,
    Name-pl = Problems ,
    name-sg = problems ,
  type = construction ,
    Name-sg = Construction ,
    name-sg = construction ,
    Name-pl = Constructions ,
    name-pl = constructions ,
  type = lemma ,
Name-pl = Lemmata ,
    name-pl = lemmata ,
}
\zcLanguageSetup{dutch}{
  type = commentary ,
    gender = f ,
    Name-sg = Toelichting ,
    name-sg = toelichting ,
    Name-pl = Toelichtingen ,
    name-pl = toelichtingen ,
  type = claim ,
    gender = m ,
    Name-sg = Claim ,
    name-sg = claim ,
    Name-pl = Claims ,
    name-pl = claims ,
  type = axiom ,
    gender = n ,
    Name-sg = Axioma ,
    name-sg = axioma ,
    Name-pl = Axiomata ,
    name-pl = axiomata ,
  type = notation ,
    gender = f ,
    Name-sg = Notatie ,
    name-sg = notatie ,
    Name-pl = Notaties ,
    name-pl = notaties ,
  type = convention ,
    gender = f ,
    Name-sg = Conventie ,
    name-sg = conventie ,
    Name-pl = Conventies ,
    name-pl = conventies ,
  type = terminology,
    gender = f ,
    Name-sg = Terminologie ,
    name-sg = terminologie ,
    Name-pl = Terminologieën ,
    name-pl = terminologieën ,
  type = problem ,
    gender = n ,
    Name-sg = Probleem ,
    name-sg = probleem ,
    Name-pl = Problemen ,
    name-sg = problemen ,
  type = construction ,
    gender = f ,
    Name-sg = Constructie ,
    name-sg = constructie ,
    Name-pl = Constructies ,
    name-pl = constructies ,
  type = lemma ,
Name-pl = Lemmata ,
    name-pl = lemmata ,
}

\zcLanguageSetup{german}{
  type = commentary,
    gender = f ,
    case = N,
      Name-sg = Erläuterung ,
      Name-pl = Erläuterungen ,
    case = G ,
      Name-sg = Erläuterung ,
      Name-pl = Erläuterungen ,
    case = D ,
      Name-sg = Erläuterung ,
      Name-pl = Erläuterungen ,
    case = A ,
      Name-sg = Erläuterung ,
      Name-pl = Erläuterungen ,
  type = claim,
    gender = m ,
    case = N,
      Name-sg = Claim ,
      Name-pl = Claims ,
    case = G ,
      Name-sg = Claims ,
      Name-pl = Claims ,
    case = D ,
      Name-sg = Claim ,
      Name-pl = Claims ,
    case = A ,
      Name-sg = Claim ,
      Name-pl = Claims ,
  type = axiom,
    gender = n ,
    case = N,
      Name-sg = Axiom ,
      Name-pl = Axiome ,
    case = G ,
      Name-sg = Axioms ,
      Name-pl = Axiome ,
    case = D ,
      Name-sg = Axiom ,
      Name-pl = Axiomen ,
    case = A ,
      Name-sg = Axiom ,
      Name-pl = Axiome ,
  type = notation,
    gender = f ,
    case = N,
      Name-sg = Notation ,
      Name-pl = Notationen ,
    case = G ,
      Name-sg = Notation ,
      Name-pl = Notationen ,
    case = D ,
      Name-sg = Notation ,
      Name-pl = Notationen ,
    case = A ,
      Name-sg = Notation ,
      Name-pl = Notationen ,
  type = convention,
    gender = f ,
    case = N,
      Name-sg = Konvention ,
      Name-pl = Konventionen ,
    case = G ,
      Name-sg = Konvention ,
      Name-pl = Konventionen ,
    case = D ,
      Name-sg = Konvention ,
      Name-pl = Konventionen ,
    case = A ,
      Name-sg = Konvention ,
      Name-pl = Konventionen ,
  type = terminology,
    gender = f ,
    case = N ,
      Name-sg = Terminologie ,
      Name-pl = Terminologien ,
    case = G ,
      Name-sg = Terminologie ,
      Name-pl = Terminologien ,
    case = D ,
      Name-sg = Terminologie ,
      Name-pl = Terminologien ,
    case = A ,
      Name-sg = Terminologie ,
      Name-pl = Terminologien ,
  type = problem,
    gender = n ,
    case = N ,
      Name-sg = Problem ,
      Name-pl = Probleme ,
    case = G ,
      Name-sg = Problems ,
      Name-pl = Probleme ,
    case = D ,
      Name-sg = Problem ,
      Name-pl = Problemen ,
    case = A ,
      Name-sg = Problem ,
      Name-pl = Probleme ,
  type = construction,
    gender = f ,
    case = N ,
      Name-sg = Konstruktion ,
      Name-pl = Konstruktionen ,
    case = G ,
      Name-sg = Konstruktion ,
      Name-pl = Konstruktionen ,
    case = D ,
      Name-sg = Konstruktion ,
      Name-pl = Konstruktionen ,
    case = A ,
      Name-sg = Konstruktion ,
      Name-pl = Konstruktionen ,
}

\iftoggle{book}{
    \newtheorem{theorem}{Theorem}[chapter]
}{
    \newtheorem{theorem}{Theorem}[section]
}

\newtheorem{corollary}[theorem]{Corollary}
\newtheorem{lemma}[theorem]{Lemma}
\newtheorem{proposition}[theorem]{Proposition}

\theoremstyle{definition}

\newtheorem{construction}[theorem]{Construction}

\newtheorem{definition}[theorem]{Definition}

\theoremstyle{remark}

\newtheorem{remark}[theorem]{Remark}

\AddToHook{env/corollary/begin}{\zcsetup{countertype={theorem=corollary}}}
\AddToHook{env/lemma/begin}{\zcsetup{countertype={theorem=lemma}}}
\AddToHook{env/proposition/begin}{\zcsetup{countertype={theorem=proposition}}}

\AddToHook{env/axiom/begin}{\zcsetup{countertype={theorem=axiom}}}
\AddToHook{env/construction/begin}{\zcsetup{countertype={theorem=construction}}}
\AddToHook{env/convention/begin}{\zcsetup{countertype={theorem=convention}}}
\AddToHook{env/definition/begin}{\zcsetup{countertype={theorem=definition}}}
\AddToHook{env/example/begin}{\zcsetup{countertype={theorem=example}}}
\AddToHook{env/notation/begin}{\zcsetup{countertype={theorem=notation}}}
\AddToHook{env/problem/begin}{\zcsetup{countertype={theorem=problem}}}
\AddToHook{env/terminology/begin}{\zcsetup{countertype={theorem=terminology}}}

\AddToHook{env/claim/begin}{\zcsetup{countertype={theorem=claim}}}
\AddToHook{env/commentary/begin}{\zcsetup{countertype={theorem=commentary}}}
\AddToHook{env/remark/begin}{\zcsetup{countertype={theorem=remark}}}

\NewTranslationFallback{Case}{Case}
\NewTranslation{dutch}{Case}{Geval}
\NewTranslation{english}{Case}{Case}
\NewTranslation{german}{Case}{Fall}

\newlist{proofcases}{enumerate}{2}
\setlist[proofcases,1]{wide=0.5em,label=\emph{\protect\GetTranslation{Case}~\arabic*.}}
\setlist[proofcases,2]{wide=1em,label*=\emph{\alph*}}

\zcsetup{
    countertype = {
        proofcasesi = case ,
        proofcasesii = case ,
    } ,
    counterresetby = {
        proofcasesii = proofcasesi ,
    }
}

\newlist{thmenumerate}{enumerate}{2}
\setlist[thmenumerate,1]{label=(\roman*), font=\normalfont}
\setlist[thmenumerate,2]{label=(\alph*), font=\normalfont}

\zcsetup{
    countertype = {
        thmenumeratei = item ,
        thmenumerateii = item ,
    } ,
    counterresetby = {
        thmenumerateii = thmenumeratei ,
    }
}

\NewDocumentEnvironment{subproof}{o}{\IfValueTF{#1}{\begin{proof}[#1]}{\begin{proof}[Pf]}}{\end{proof}}

\let\lleft\mleft
\let\rright\mright
\let\mmiddle\middle

\NewDocumentCommand{\mand}{}{\quad \text{and} \quad}

\NewDocumentCommand\lfalse{}{\bot}
\NewDocumentCommand\ltrue{}{\top}
\NewDocumentCommand\limplies{}{\DOTSB\rightarrow}
\NewDocumentCommand\liff{}{\DOTSB\leftrightarrow}
\NewDocumentCommand\ldiamond{}{\mdwhtdiamond}
\NewDocumentCommand\lbox{}{\mdwhtsquare}
\NewDocumentCommand\lcircle{}{\mdwhtcircle}

\NewDocumentCommand\suchthat{}{\;\,}
\NewDocumentCommand\suchthatshort{}{\;}

\ExplSyntaxOn
\keys_define:nn { my/math/quantifier }
    {
        quantifier .tl_set:N = \l__my_math_quantifier_quant_tl,
        quantifier .value_required:n = true,
        exists .meta:n = {quantifier = {\exists}},
        forall .meta:n = {quantifier = {\forall}},
        doparen .bool_set:N = \l__my_math_quantifier_doparen_bool,
        doparen .default:n = true,
        nextquant .bool_set:N = \l__my_math_quantifier_nextquant_bool,
        nextquant .default:n = true,
    }

\keys_precompile:nnN { my/math/quantifier } { quantifier = {}, doparen = false, nextquant = false }
\l__my_math_quantifier_reset_defaults_tl

\NewDocumentCommand{\MakeQuantifier}{m m m}
    {
        \group_begin:

        \tl_use:N \l__my_math_quantifier_reset_defaults_tl

        \keys_set:nn { my/math/quantifier } { #1 }

        \bool_if:NTF \l__my_math_quantifier_nextquant_bool
        { \tl_use:N \l__my_math_quantifier_quant_tl #2 \suchthatshort #3 }
        {
            \bool_if:NTF \l__my_math_quantifier_doparen_bool
            { \tl_use:N \l__my_math_quantifier_quant_tl #2 \suchthatshort\left\lbrack #3 \right\rbrack }
            { \tl_use:N \l__my_math_quantifier_quant_tl #2 . \suchthat #3 }
        }

        \group_end:
    }

\NewDocumentCommand{\ForAll}{o m m}
    {
        \IfValueTF { #1 }
        { \MakeQuantifier{ #1, forall }{ #2 }{ #3 } }
        { \MakeQuantifier{ forall }{ #2 }{ #3 } }
    }

\NewDocumentCommand{\Exists}{o m m}
    {
        \IfValueTF { #1 }
        { \MakeQuantifier{ #1, exists }{ #2 }{ #3 } }
        { \MakeQuantifier{ exists }{ #2 }{ #3 } }
    }
\ExplSyntaxOff

\NewDocumentCommand{\Set}{m o}{\IfValueTF{#2}{\lleft\{#1\;\mmiddle|\;#2\rright\}}{\lleft\{#1\rright\}}}

\NewDocumentCommand{\Function}{m m m}{#1\colon #2 \to #3}

\NewVariableClass\mClassDelim

\NewObject\mClassDelim\paren[left par=\lparen, right par=\rparen]

\NewObject\mClassDelim\tuple[left par=\langle, right par=\rangle]
\NewObject\mClassDelim\homotuple[left par=\langle, right par=\rangle]
\NewObject\mClassDelim\hetrotuple[left par=\langle, right par=\rangle]
\NewObject\mClassDelim\structuple[left par=\langle\nobreak, right par=\rangle]

\NewObject\mClassDelim\abs[left par=\lvert, right par=\rvert]
\NewObject\mClassDelim\norm[left par=\lVert , right par=\rVert]
\NewObject\mClassDelim\card[left par=\lvert, right par=\rvert]
\NewObject\mClassDelim\sem[left par=\lBrack, right par=\rBrack]

\NewVariableClass\mClassMath

\NewDocumentCommand\mathstylealg{m}{\symbf{#1}}
\NewDocumentCommand\mathstylea{m}{#1}
\NewDocumentCommand\mathstylecd{m}{#1}
\NewDocumentCommand\mathstylef{m}{#1}
\NewDocumentCommand\mathstylefr{m}{\symfrak{#1}}
\NewDocumentCommand\mathstylefrg{m}{\symfrak{#1}}
\NewDocumentCommand\mathstylefrl{m}{\symfrak{#1}}
\NewDocumentCommand\mathstylefrm{m}{\symfrak{#1}}
\NewDocumentCommand\mathstylelog{m}{#1}
\NewDocumentCommand\mathstylel{m}{#1}
\NewDocumentCommand\mathstylelth{m}{#1}
\NewDocumentCommand\mathstylen{m}{#1}
\NewDocumentCommand\mathstyleod{m}{#1}
\NewDocumentCommand\mathstylereal{m}{#1}
\NewDocumentCommand\mathstylerel{m}{#1}
\NewDocumentCommand\mathstyles{m}{#1}
\NewDocumentCommand\mathstyletop{m}{\symfrak{#1}}

\NewObject\mClassMath\algA{\mathstylealg{A}}
\NewObject\mClassMath\algB{\mathstylealg{B}}
\NewObject\mClassMath\algC{\mathstylealg{C}}
\NewObject\mClassMath\algFree{\mathstylealg{Fr}}
\NewObject\mClassMath\algBAtwo{\mathstylealg{2}}
\NewObject\mClassMath\algMinTropical{\mathstylealg{T}}
\NewObject\mClassMath\algMinTropicalExt{\bar{\mathstylealg{T}}}

\DeclareObject\mClassMath\aa{\mathstylea{a}} \NewObject\mClassMath\ab{\mathstylea{b}}
\NewObject\mClassMath\ac{\mathstylea{c}}

\DeclareObject\mClassMath\cdkappa{\mathstylecd{\kappa}}
\DeclareObject\mClassMath\cdlambda{\mathstylecd{\lambda}}

\NewObject\mClassMath\fe{\mathstylef{e}}
\NewObject\mClassMath\ff{\mathstylef{f}}
\NewObject\mClassMath\fg{\mathstylef{g}}
\NewObject\mClassMath\fh{\mathstylef{h}}
\NewObject\mClassMath\fk{\mathstylef{k}}
\NewObject\mClassMath\fl{\mathstylef{l}}
\NewObject\mClassMath\fp{\mathstylef{p}}
\NewObject\mClassMath\fq{\mathstylef{q}}
\NewObject\mClassMath\fr{\mathstylef{r}}
\NewObject\mClassMath\fs{\mathstylef{s}}
\NewObject\mClassMath\ft{\mathstylef{t}}
\NewObject\mClassMath\fu{\mathstylef{u}}
\NewObject\mClassMath\fF{\mathstylef{F}}
\NewObject\mClassMath\fG{\mathstylef{G}}
\NewObject\mClassMath\fH{\mathstylef{H}}
\NewObject\mClassMath\fiota{\mathstylef{\iota}}
\NewObject\mClassMath\fsigma{\mathstylef{\sigma}}
\NewObject\mClassMath\fSeqShift{\mathstylef{\sigma}}
\SemantexRecordObject{\fComp}
\NewDocumentCommand{\fComp}{m}{
    \SemantexRecordSource { \fComp { #1 } }
    \paren { #1 }
}
\SemantexRecordObject{\fDom}
\NewDocumentCommand{\fDom}{m}{
    \SemantexRecordSource { \fDom { #1 } }
    \operatorname{dom}\paren{ #1 }
}
\SemantexRecordObject{\fIm}
\NewDocumentCommand{\fIm}{m}{
    \SemantexRecordSource { \fIm { #1 } }
    \operatorname{Im}\paren{ #1 }
}
\SemantexRecordObject{\fSupp}
\NewDocumentCommand{\fSupp}{m}{
    \SemantexRecordSource { \fSupp { #1 } }
    \operatorname{supp}\paren{ #1 }
}

\NewObject\mClassMath\fVal{\symfrak{V}}
\NewObject\mClassMath\fSpan{\mathrm{sp}}
\NewObject\mClassMath\fTy{\mathrm{ty}}
\NewObject\mClassMath\fClust{\mathrm{clu}}

\NewObject\mClassMath\frF{\mathstylefr{F}}
\NewObject\mClassMath\frG{\mathstylefr{G}}
\NewObject\mClassMath\frH{\mathstylefr{H}}
\NewObject\mClassMath\frT{\mathstylefr{T}}
\NewObject\mClassMath\frCluster{\mathstylefr{C}}

\NewObject\mClassMath\frgF{\mathstylefrg{f}}
\NewObject\mClassMath\frgG{\mathstylefrg{g}}
\NewObject\mClassMath\frgH{\mathstylefrg{h}}

\NewObject\mClassMath\frlQ{\mathstylefrl{Q}}
\NewObject\mClassMath\frlS{\mathstylefrl{S}}
\NewObject\mClassMath\frlT{\mathstylefrl{T}}

\NewObject\mClassMath\frmM{\mathstylefrm{M}}
\NewObject\mClassMath\frmN{\mathstylefrm{N}}

\NewObject\mClassMath\lBox{\lbox}
\NewObject\mClassMath\lDiamond{\ldiamond}
\NewObject\mClassMath\lNext{\lcircle}

\NewObject\mClassMath\logLambda{\mathstylelog{\Lambda}}

\NewObject\mClassMath\logForm{\mathrm{Form}}
\NewObject\mClassMath\logCPC{\mathstylelog{\mathbf{CPC}}}

\NewObject\mClassMath\logSiIPC{\mathstylelog{\mathbf{IPC}}}
\NewObject\mClassMath\logSiKC{\mathstylelog{\mathbf{KC}}}
\NewObject\mClassMath\logSiLPtwo{\mathstylelog{\mathbf{LP}_2}}
\NewObject\mClassMath\logSiLV{\mathstylelog{\mathbf{LV}}}
\NewObject\mClassMath\logSiLC{\mathstylelog{\mathbf{LC}}}
\NewObject\mClassMath\logSiLS{\mathstylelog{\mathbf{LS}}}

\NewObject\mClassMath\logK{\mathbf{K}}
\NewObject\mClassMath\logKWeakTrans{\mathbf{wK4}}
\NewObject\mClassMath\logKWeakTransLin{\mathbf{wK4.3}}
\NewObject\mClassMath\logKTrans{\mathbf{K4}}
\NewObject\mClassMath\logKTransLin{\mathbf{K4.3}}
\SemantexRecordObject{\logKTransBoundedWidth}
\DeclareDocumentCommand{\logKTransBoundedWidth}{m}{\SemantexRecordSource {\logKTransBoundedWidth {#1}}\UseClassInCommand\mClassMath{\mathbf{K4BW}_{#1}}}
\NewObject\mClassMath\logKRefl{\mathbf{KT}}
\NewObject\mClassMath\logKReflTrans{\mathbf{S4}}
\NewObject\mClassMath\logKReflTransLin{\mathbf{S4.3}}
\NewObject\mClassMath\logKReflTransLinFinal{\mathbf{S4.3.1}}
\NewObject\mClassMath\logKReflTransConfl{\mathbf{S4.2}}
\NewObject\mClassMath\logKReflTransConflFinal{\mathbf{S4.2.1}}
\NewObject\mClassMath\logKReflTransFinal{\mathbf{S4.1}}
\SemantexRecordObject{\logKReflTransBoundedWidth}
\DeclareDocumentCommand{\logKReflTransBoundedWidth}{m}{\SemantexRecordSource {\logKReflTransBoundedWidth {#1}}\UseClassInCommand\mClassMath{\mathbf{S4BW}_{#1}}}
\NewObject\mClassMath\logKReflTransSym{\mathstylelog{\mathbf{S5}}}
\NewObject\mClassMath\logGrz{\mathbf{Grz}}
\NewObject\mClassMath\logGrzConfl{\mathbf{Grz.2}}
\NewObject\mClassMath\logGrzLin{\mathbf{Grz.3}}
\SemantexRecordObject{\logGrzBoundedWidth}
\DeclareDocumentCommand{\logGrzBoundedWidth}{m}{\SemantexRecordSource {\logGrzBoundedWidth {#1}}\UseClassInCommand\mClassMath{\mathbf{GrzBW}_{#1}}}
\NewObject\mClassMath\logGl{\mathbf{GL}}
\NewObject\mClassMath\logGlLin{\mathbf{GL.3}}
\SemantexRecordObject{\logGlBoundedWidth}
\DeclareDocumentCommand{\logGlBoundedWidth}{m}{\SemantexRecordSource {\logGlBoundedWidth {#1}}\UseClassInCommand\mClassMath{\mathbf{GLBW}_{#1}}}

\NewObject\mClassMath\lalpha{\mathstylel{\alpha}}
\NewObject\mClassMath\lbeta{\mathstylel{\beta}}
\NewObject\mClassMath\lgamma{\mathstylel{\gamma}}
\NewObject\mClassMath\ldelta{\mathstylel{\delta}}
\NewObject\mClassMath\lepsilon{\mathstylel{\epsilon}}
\NewObject\mClassMath\leta{\mathstylel{\eta}}
\NewObject\mClassMath\lphi{\mathstylel{\varphi}}
\NewObject\mClassMath\lpsi{\mathstylel{\psi}}
\NewObject\mClassMath\lchi{\mathstylel{\chi}}
\NewObject\mClassMath\lFrame{\mathstylel{\chi}}
\NewObject\mClassMath\lSubFrame{\mathstylel{\beta}}
\NewObject\mClassMath\lCofSubFrame{\mathstylel{\gamma}}

\NewObject\mClassMath\lap{\mathstylel{p}}
\NewObject\mClassMath\laq{\mathstylel{q}}
\NewObject\mClassMath\lar{\mathstylel{r}}

\NewObject\mClassMath\lthS{\mathstylelth{S}}
\NewObject\mClassMath\lthT{\mathstylelth{T}}
\NewObject\mClassMath\lthGamma{\mathstylelth{\Gamma}}
\NewObject\mClassMath\lthDelta{\mathstylelth{\Delta}}
\NewObject\mClassMath\lthSigma{\mathstylelth{\Sigma}}
\NewObject\mClassMath\lthPhi{\mathstylelth{\Phi}}
\NewObject\mClassMath\lthPsi{\mathstylelth{\Psi}}
\NewObject\mClassMath\lthSubform{\mathrm{SubF}}\NewObject\mClassMath\lthProp{\symbb{P}}

\NewObject\mClassMath\nii{\mathstylen{i}}
\NewObject\mClassMath\nj{\mathstylen{j}}
\NewObject\mClassMath\nk{\mathstylen{k}}
\NewObject\mClassMath\nl{\mathstylen{l}}
\NewObject\mClassMath\nm{\mathstylen{m}}
\NewObject\mClassMath\nn{\mathstylen{n}}
\NewObject\mClassMath\nN{\mathstylen{N}}
\NewObject\mClassMath\np{\mathstylen{p}}

\NewObject\mClassMath\odalpha{\mathstyleod{\alpha}}
\NewObject\mClassMath\odbeta{\mathstyleod{\beta}}
\NewObject\mClassMath\odgamma{\mathstyleod{\gamma}}
\NewObject\mClassMath\oddelta{\mathstyleod{\delta}}
\NewObject\mClassMath\odlambda{\mathstyleod{\lambda}}

\NewObject\mClassMath\realr{\mathstylereal{r}}
\NewObject\mClassMath\realalpha{\mathstylereal{\alpha}}
\NewObject\mClassMath\realbeta{\mathstylereal{\beta}}
\NewObject\mClassMath\realgamma{\mathstylereal{\gamma}}
\NewObject\mClassMath\reallambda{\mathstylereal{\lambda}}

\NewObject\mClassMath\relD{\mathstylerel{D}}
\NewObject\mClassMath\relE{\mathstylerel{E}}
\NewObject\mClassMath\relF{\mathstylerel{F}}
\NewObject\mClassMath\relG{\mathstylerel{G}}
\NewObject\mClassMath\relR{\mathstylerel{R}}
\NewObject\mClassMath\relS{\mathstylerel{S}}
\NewObject\mClassMath\relZ{\mathstylerel{Z}}
\NewObject\mClassMath\relEquiv{\equiv}
\SetupObject\relEquiv{
    define keys={
        {not}{symbol=\nequiv},
    },
}
\NewObject\mClassMath\relPreEquiv{\sqsubseteq}
\SetupObject\relPreEquiv{
    define keys={
        {not}{symbol=\nsqsubseteq},
        {strict}{symbol=\sqsubset},
    },
}
\NewObject\mClassMath\relPre{\preceq}
\SetupObject\relPre{
    define keys={
        {not}{symbol=\mathrel{\not\preceq}},
        {strict}{symbol=\prec},
    },
}
\SemantexRecordObject{\relComp}
\NewDocumentCommand{\relComp}{m}{
    \SemantexRecordSource { \relComp { #1 } }
    \paren { #1 }
}
\NewObject\mClassMath\relBisim{\underline{\leftrightarrow}}

\NewObject\mClassMath\sEmpty{\emptyset}
\NewObject\mClassMath\sA{\mathstyles{A}}
\NewObject\mClassMath\sB{\mathstyles{B}}
\NewObject\mClassMath\sC{\mathstyles{C}}
\NewObject\mClassMath\sD{\mathstyles{D}}
\NewObject\mClassMath\sE{\mathstyles{E}}
\NewObject\mClassMath\sF{\mathstyles{F}}
\NewObject\mClassMath\sG{\mathstyles{G}}
\NewObject\mClassMath\sH{\mathstyles{H}}
\NewObject\mClassMath\sI{\mathstyles{I}}
\NewObject\mClassMath\sJ{\mathstyles{J}}
\NewObject\mClassMath\sK{\mathstyles{K}}
\NewObject\mClassMath\sL{\mathstyles{L}}
\NewObject\mClassMath\sM{\mathstyles{M}}
\NewObject\mClassMath\sN{\mathstyles{N}}
\NewObject\mClassMath\sO{\mathstyles{O}}
\NewObject\mClassMath\sP{\mathstyles{P}}
\NewObject\mClassMath\sS{\mathstyles{S}}
\NewObject\mClassMath\sT{\mathstyles{T}}
\NewObject\mClassMath\sU{\mathstyles{U}}
\NewObject\mClassMath\sV{\mathstyles{V}}
\NewObject\mClassMath\sW{\mathstyles{W}}
\NewObject\mClassMath\sX{\mathstyles{X}}
\NewObject\mClassMath\sY{\mathstyles{Y}}
\NewObject\mClassMath\sZ{\mathstyles{Z}}
\NewObject\mClassMath\sTwo{\symbb{2}}
\NewObject\mClassMath\sTypes{\mathrm{Ty}}
\NewObject\mClassMath\sTopClust{\mathrm{TopClust}}
\NewObject\mClassDelim\sEquivClass[left par=\lbrack, right par=\rbrack]
\NewObject\mClassDelim\sNumber[left par=\lbrack, right par=\rbrack]

\NewObject\mClassMath\sPowerSet{\mathcal{P}}

\NewObject\mClassMath\topF{\mathstyletop{F}}
\NewObject\mClassMath\topG{\mathstyletop{G}}
\NewObject\mClassMath\topX{\mathstyletop{X}}
\NewObject\mClassMath\topY{\mathstyletop{Y}}
\NewObject\mClassMath\topZ{\mathstyletop{Z}}

\NewObject\mClassMath\vi{i}
\NewObject\mClassMath\vr{r}
\NewObject\mClassMath\vs{s}
\NewObject\mClassMath\vt{t}
\NewObject\mClassMath\vu{u}
\NewObject\mClassMath\vv{v}
\NewObject\mClassMath\vw{w}
\NewObject\mClassMath\vx{x}
\NewObject\mClassMath\vy{y}
\NewObject\mClassMath\vz{z}

\NewObject\mClassMath\Naturals{\symbb{N}}
\SetupObject\Naturals{
    define keys={
{strict pos}{star},
    },
}
\NewObject\mClassMath\Integers{\symbb{Z}}
\NewObject\mClassMath\Rationals{\symbb{Q}}
\NewObject\mClassMath\Reals{\symbb{R}}
\NewObject\mClassMath\ExtReals{\overline{\symbb{R}}}

\NewSymbolClass\mClassOperators

\NewObject\mClassOperators\fBinProd{\times}

\NewObject\mClassOperators\aSRplus{\oplus}
\NewObject\mClassOperators\aSRtimes{\otimes}
\NewObject\mClassOperators\aSRzero{0}
\NewObject\mClassOperators\aSRone{1}

\NewObject\mClassOperators\aLGmeet{\wedge}
\NewObject\mClassOperators\aLGjoin{\vee}
\NewObject\mClassOperators\aLGplus{+}
\NewObject\mClassOperators\aLGminus{-}
\NewObject\mClassOperators\aLGzero{e}
\NewObject\mClassOperators\aLGbottom{\bot}
\NewObject\mClassOperators\aLGtop{\top}

\NewObject\mClassOperators\sMinus{\setminus}
\NewObject\mClassOperators\sCartProd{\times}
\NewObject\mClassOperators\sUnion{\cup}
\NewObject\mClassOperators\sDisjointUnion{\amalg}
\NewObject\mClassOperators\sIntersection{\cap}
\NewObject\mClassOperators\sSetUnion{\bigcup\:}
\NewObject\mClassOperators\sSetDisjointUnion{\coprod\:}
\NewObject\mClassOperators\sSetIntersection{\bigcap\:}
\NewObject\mClassOperators\sLimitUnion{\bigcup}
\NewObject\mClassOperators\sLimitDisjointUnion{\coprod}
\NewObject\mClassOperators\sLimitIntersection{\bigcap}
\NewObject\mClassOperators\sSubsetEq{\subseteq}
\NewObject\mClassOperators\sSubsetStrict{\subsetneq}
\NewObject\mClassOperators\sSupersetEq{\supseteq}
\NewObject\mClassOperators\sSupersetStrict{\supsetneq}
\NewObject\mClassOperators\sEquinumerous{\sim}

\NewObject\mClassOperators\Equal{=}
\NewObject\mClassOperators\LessStrict{<}
\NewObject\mClassOperators\LessEq{\leq}
\NewObject\mClassOperators\GreaterStrict{>}
\NewObject\mClassOperators\GreaterEq{\geq}

\NewObject\mClassOperators\join{\vee}
\NewObject\mClassOperators\meet{\wedge}
\NewObject\mClassOperators\JOIN{\bigvee}
\NewObject\mClassOperators\MEET{\bigwedge}

\NewObject\mClassOperators\lAND{\bigwedge}
\NewObject\mClassOperators\lOR{\bigvee}

\NewObject\mClassOperators\lProves{\vdash}
\NewObject\mClassOperators\lNotProves{\nvdash}
\NewObject\mClassOperators\lValidates{\vDash}
\NewObject\mClassOperators\lNotValidates{\nvDash}
\NewObject\mClassOperators\Isomorphic{\cong}
\NewObject\mClassOperators\NotIsomorfic{\ncong}
\NewObject\mClassOperators\ElementaryEquiv{\equiv}

\NewObject\mClassOperators\Sum{\sum}

\NewObject\mClassOperators\vequiv{\sim}

\SetupClass\mClassDelim{
    output=\mClassMath,
    define keys={
        {prime}{output options={prime}},
        {'}{output options={'}},
        {''}{output options={''}},
        {'''}{output options={'''}},
        {fn}{set arg single keys={\cdot}},
    },
    define keys[1]={
        {default}{output options={default={#1}}},
        {lower}{output options={lower={#1}}},
    },
}

\SetupClass\mClassMath{
    output=\mClassMath,
    define keys={
        {vec}{command=\vec},
        {tilde}{command=\tilde},
        {widetilde}{command=\widetilde},
        {hat}{command=\hat},
        {star}{upper=\ast},
        {overline}{command=\overline},
        {box}{lower={\lBox}},
        {converse}{upper=\mathrm{op}},
        {inv}{upper={-1}},
        {preimage}{upper={-1}},
        {gen}{overline},
        {complex alg}{upper=\ast},
        {negative}{upper={\leq 0}},
        {Delta}{upper={\Delta}},
        {refl closure}{overline},
        {trans closure}{upper=+},
        {refl trans closure}{star},
        {irrefl closure}{upper=\circ},
{can}{upper=\mathrm{c}},
{domain}{command=\fDom},
        {dom}{domain},
{image}{command=\fIm},
{card}{return, command=\card},
{F}{lower={\mathrm{f}}},
{P}{lower={\mathrm{p}}},
{M}{lower={\mathrm{m}}},
{least modal comp}{return, spar, symbol put left={\tau}},
{greatest modal comp}{return, spar, symbol put left={\sigma}},
{si fragment}{return, spar, symbol put left={\rho}},
    },
    define keys[1]={
        {pow}{upper=#1},
        {restrict}{right return,symbol put right={\mathopen{}\upharpoonright\mathclose{} {#1}}},
        {quotient}{right return,symbol put right={\mathopen{}\mathslash\mathclose{} {#1}}},
{subst}{right return, symbol put right={\lbrack #1 \rbrack}},
{gen subfr}{upper=#1},
    },
}

\SetupClass\mClassOperators{
    define keys={
        {overline}{command=\overline},
    },
    define keys[1]={
        {default}{output options={default={#1}}},
        {lower}{output options={lower={#1}}},
        {upper}{output options={upper={#1}}},
    },
}

\ExplSyntaxOn

\NewDocumentCommand{\ElimCrit}{o} {
    \IfValueTF { #1 }
    {
        \group_begin:
\keys_set:nn { mymath/elimcrit } { #1 }
\symcal{E}
\tl_if_empty:NTF \l__mymath_elimcrit_model_tl {} {
            \c_math_subscript_token { \tl_use:N \l__mymath_elimcrit_model_tl }
        }
\c_math_superscript_token { \l__mymath_elimcrit_predefined_cs:n { \tl_use:N \l__mymath_elimcrit_theory_tl } }
\tl_if_empty:NTF \l__mymath_elimcrit_upset_tl {} {
            \paren{ \tl_use:N \l__mymath_elimcrit_upset_tl }
        }
        \group_end:
    }
    { \symcal{E} }
}

\keys_define:nn { mymath/elimcrit }
    {
        model .tl_set:N = \l__mymath_elimcrit_model_tl,
        model .value_required:n = true,
        upset .tl_set:N = \l__mymath_elimcrit_upset_tl,
        upset .value_required:n = true,
        theory .tl_set:N = \l__mymath_elimcrit_theory_tl,
        theory .value_required:n = true,

        predefined .choice:,
        predefined / none .code:n = \cs_set_protected_nopar:Nn \l__mymath_elimcrit_predefined_cs:n {},
        predefined / weak .code:n = \cs_set_protected_nopar:Nn \l__mymath_elimcrit_predefined_cs:n {##1\textrm{-weak}},
        predefined / strong .code:n = \cs_set_protected_nopar:Nn \l__mymath_elimcrit_predefined_cs:n {##1\textrm{-strong}},
        predefined .initial:n = none,

        weak .meta:n = {predefined=weak},
        weak .value_forbidden:n = true,
        strong .meta:n = {predefined=strong},
        strong .value_forbidden:n = true,

        m .meta:n = {model=#1},
        u .meta:n = {upset=#1},
        t .meta:n = {theory=#1},
    }

\ExplSyntaxOff

\RequirePackage{tikz-cd}
\usetikzlibrary{calc}
\usetikzlibrary{decorations.pathmorphing}
\usetikzlibrary{spath3}

\tikzset{curve/.style={settings={#1},to path={(\tikztostart)
    .. controls ($(\tikztostart)!\pv{pos}!(\tikztotarget)!\pv{height}!270:(\tikztotarget)$)
    and ($(\tikztostart)!1-\pv{pos}!(\tikztotarget)!\pv{height}!270:(\tikztotarget)$)
    .. (\tikztotarget)\tikztonodes}},
    settings/.code={\tikzset{quiver/.cd,#1}
        \def\pv##1{\pgfkeysvalueof{/tikz/quiver/##1}}},
    quiver/.cd,pos/.initial=0.35,height/.initial=0}

\tikzset{between/.style n args={2}{/tikz/spath/at end path construction={
    \tikzset{spath/split at keep middle={current}{#1}{#2}}
}}}

\tikzset{tail reversed/.code={\pgfsetarrowsstart{tikzcd to}}}
\tikzset{2tail/.code={\pgfsetarrowsstart{Implies[reversed]}}}
\tikzset{2tail reversed/.code={\pgfsetarrowsstart{Implies}}}
\tikzset{no body/.style={/tikz/dash pattern=on 0 off 1mm}}

\SetupObject\sTwo{
    symbol={\Set{0, 1}},
}

\SetupObject\logForm{
    symbol={\mathbf{Fm}},
}

\ifxetex
    \SetupObject\relPre{
        define keys={ {not}{symbol=\mathrel{\mathpalette\overlaynot\preceq}} }
    }
\fi

\NewObject\mClassMath\logModalL{\symup{L}}

\SetupClass\mClassMath{
    define keys={
        {unbox}{lower=\lbox},
    },
}

\SetupObject\frCluster{
    define keys={
        {withtop}{overline},
    },
}
\NewObject\mClassMath\lCluFr{\mathstylel{\gamma}}
\SetupObject\lCluFr{
    define keys={
        {withtop}{command=\bar},
    },
}
\NewObject\mClassMath\modcomp{\mathstylel{\tau}}

\NewObject\mClassMath\logGamma{\mathstylelog{\Gamma}}

\NewObject\mClassMath\boolsub{\mathrm{B}}
\NewDocumentCommand{\subst}{m m}{#1\lbrack#2\rbrack}
 
\newcommand*\samethanks[1][\value{footnote}]{\footnotemark[#1]}
\newcommand*{\thetitle}{Interpolation above S4}
\newcommand*{\thedate}{\today}
\author[1]{Simon Santschi\,\orcidlink{0009-0003-9364-5149}\thanks{Supported by the Swiss National Science Foundation (SNSF), grant no.\@ 200021\textunderscore215157.}\thanks{Supported by the Swiss National Science Foundation (SNSF), grant no.\@ P500-2\textunderscore242432.}}
\author[2]{Niels C.~Vooijs\,\orcidlink{0000-0002-4515-9694}\samethanks[1]}

\affil[1]{Institute of Computer Science, Czech Academy of Sciences, Prague, Czechia}
\affil[2]{Mathematical Institute, University of Bern, Bern, Switzerland}
\date{\thedate}
\title{\thetitle}

\begin{document}
    \maketitle
    \begin{abstract}
        \noindent
        We complete Maksimova's classification of Craig and deductive interpolation in normal extensions of \(\logKReflTrans\).
In particular, we prove Craig interpolation for the six extensions of \(\logKReflTrans\) for which the status of the property was still open.
        The proof strategy builds upon the ideas of \citeauthor{Smorynski1978-Beths-theorem-and-self-referential-sentences}, but employs a novel approach using Fine's frame formulas for splitting clusters in transitive models.
\end{abstract}

    \section{Introduction}

By a celebrated result of Maksimova \cite{Maksimova1978-Craigs-theorem}, exactly eight superintuitionistic logics have the Craig interpolation property (\tCIP{}).
Shortly thereafter, Maksimova showed, using the intimate connection between superintuitionistic logics and extensions of the modal logic \(\logKReflTrans\), that at most 38 normal extensions of \(\logKReflTrans\) have the \tCIP{} \cite{Maksimova1980-interpolation-theorems}, and reduced this to 37 in \cite{Maksimova1984-absence-CIP-Dummett}.
Moreover, in \cite{Maksimova1981-interpolation-theorems-sufficient-conditions}, she proved the \tCIP{} for the majority of these logics, more specifically, the ones that are canonical.
Subsequently, Boolos \cite{Boolos1980-systems-of-modal-logic-with-provability-interpretations}, Rautenberg \cite{Rautenberg1983-modal-tableau-calculi-and-interpolation}, Maksimova \cite{Maksimova1987-interpolation-in-normal-modal-logics}, and \textcite{Shimura1992-cut-free-systems-for-some-modal-logics-containing-S4}\footnote{\Textcite{Shimura1992-cut-free-systems-for-some-modal-logics-containing-S4} proved the \tCIP{} for two logics. \Textcite{Maksimova1991-amalgamation-and-interpolation-in-normal-modal-logics} already claimed these logics have the \tCIP{}, but to the best of our knowledge, the proof was never published.} proved the \tCIP{} for, in total, another six of the logics.
For the final six logics, the status of the \tCIP{} has remained open.
Regarding the deductive interpolation property (\tDIP{}), the story is similar, with at most 49 normal extensions of \(\logKReflTrans\) enjoying the \tDIP{}, and the six open cases coinciding exactly with those for the \tCIP{}.
A complete overview of the situation is presented in \cite{GM2005}.

A logic \(\logModalL\) is said to have the \emph{Craig interpolation property} (or \emph{\titCIP{}} for short) if for all formulas \(\lphi,\lpsi\) with \(\lphi\limplies \lpsi \in \logModalL\), there exists a formula \(\lchi\) whose atomic propositions appear in both \(\lphi\) and \(\lpsi\), such that \(\lphi \limplies \lchi \in \logModalL\) and \(\lchi\limplies \lpsi \in \logModalL\). In this case the formula \(\lchi\) is called an \emph{interpolant} for the implication \(\lphi\limplies \lpsi\).

\Textcite{Maksimova1980-interpolation-theorems} showed that if a normal extension of \(\logKReflTrans\) has the \tCIP{}, then its intuitionistic fragment also has the \tCIP{}.
This restricts the scope for extensions of \(\logKReflTrans\) with the \tCIP{} to modal companions of the eight superintuitionistic logics with the \tCIP{}.

For this reason Maksimova introduced the following notation: for \(\nm,\nn \in \Naturals \sUnion \Set{\omega}\) and a superintuitionistic logic \(\logLambda\), \(\Gamma(\logLambda,\nm,\nn)\) denotes the normal extension of \(\logKReflTrans\) that is complete with respect to the finite \tframes{\modcomp{\logLambda}} for which the final clusters have size at most \(\nm\) and the non-final clusters have size at most \(\nn\), where \(\modcomp{\logLambda}\) denotes the smallest modal companion of \(\logLambda\).
For example, \(\Gamma(\logSiIPC,\omega,\omega) = \logKReflTrans\), \(\Gamma(\logSiIPC,1,1) = \logGrz\), \(\Gamma(\logSiKC,\omega,\omega) = \logKReflTransConfl\), and \(\Gamma(\logCPC, \omega, \nn) = \logKReflTransSym\) for any \(\nn\), where \(\logSiIPC\) denotes intuitionistic propositional logic, \(\logCPC\) classical propositional logic, and \(\logSiKC = \logSiIPC + \paren{\lnot\lap \lor \lnot\lnot\lap}\).

The at most 37 extensions of \(\logKReflTrans\) enjoying the \tCIP{} are all of the form \(\Gamma(\logLambda,\nm,\nn)\), where \(\logLambda\) has the \tCIP{} and \(\nm,\nn \in \Set{1,2,\omega}\).
As already mentioned, 31 of these logics are known to have the \tCIP{}.
The six normal extensions of \(\logKReflTrans\) for which the \tCIP{} remained open, are
\begin{equation*}
	\Gamma(\logLambda, \nm, 2) \qquad \text{for } \logLambda \in \Set{\logSiIPC, \logSiKC} \; \text{ and } \; \nm \in \Set{1, 2, \omega},
\end{equation*}
see \cite[pp.~256\textendash{}257]{GM2005}.
We show that these six remaining logics have the \tCIP{}. In fact, we give a uniform proof of the \tCIP{} for the logics \(\Gamma(\logLambda, \nm, \nn)\) with \(\logLambda \in \Set{\logSiIPC, \logSiKC}\) and \(\nm,\nn \in \Set{1,2,\omega}\)\footnote{\textcite[first corollary to Theorem~4.4]{Shimura1992-cut-free-systems-for-some-modal-logics-containing-S4} proved that if a logic of the form $\Gamma(\logSiIPC, \nm, \nn)$ has the  \tCIP{}, then so does $\Gamma(\logSiKC, \nm, \nn)$. However,  for the sake of being self-contained, and since our approach works equally well for modal companions of $\logSiKC$, we decided not to use this result.}.

Since for modal logics, the \tCIP{} implies the \tDIP{}, this completes Maksimova's characterization for both the \tCIP{} and the \tDIP{} in normal extensions of \(\logKReflTrans\).
This resolves Problem~14.3 of \cite{CZ1997}, which was recently reposed in \cite[Open Problem~9]{Handbook2026-Craig-interpolation-chapter-3}.

\paragraph{Method.}
Our approach is based on the technique that \textcite{Smorynski1978-Beths-theorem-and-self-referential-sentences} used to prove the \tCIP{} for \(\logGl\).
This approach was also used by \textcite{Boolos1980-systems-of-modal-logic-with-provability-interpretations} to prove the \tCIP{} for \(\logGrz\) and by \textcite{Maksimova1987-interpolation-in-normal-modal-logics} for two further normal extensions of \(\logKReflTrans\).

Smoryński's method relies on syntactically constructing finite models.
First it is assumed that an implication does not have an interpolant.
Using this, a finite countermodel for the implication is constructed (see \zcref{sec:smorynski}).
For logics that impose restrictions on the clusters in their frames,
it is necessary to tweak the definition of the accessibility relation so that the constructed model is a model of the logic, cf.\@ \cite{Smorynski1978-Beths-theorem-and-self-referential-sentences,Boolos1980-systems-of-modal-logic-with-provability-interpretations,Maksimova1987-interpolation-in-normal-modal-logics}.
Thus, every implication lacking an interpolant is invalid in the logic, yielding the \tCIP{}.

The key difference of our approach compared to the previous ones is that we employ a more flexible, semantic approach towards abiding by the restrictions on clusters.
We first construct a standard Smoryński model \(\frmM\), which is only a model for \(\logKReflTrans\) or \(\logKReflTransConfl\).
The cluster restrictions are then expressed using Fine's frame formulas \cite{Fine1974-ascending-chain-S4-logics}, and it follows from the Truth Lemma for Smoryński models that \(\frmM\) validates sufficiently large substitutions of these frame formulas.
Using these, we show that every cluster contains a small enough \emph{adequate} subset of points, which together do not rely on other points in the cluster to witnesses the satisfaction of formulas.
This finally allows us to modify the accessibility relation into one that does not contain clusters that are too large, while preserving the satisfaction of enough formulas.

\paragraph{Outline.}
The paper is structured as follows.
\zcref{sec:prelims} contains the preliminaries needed for the rest of the paper. In particular, we recall the definition and fundamental property of subframe formulas and the definition of the relevant extensions of \(\logKReflTrans\).
In \zcref{sec:smorynski}, we present Smoryński models and consider their basic properties.
In \zcref{sec:refine-clusters}, we establish a way to reduce the size of clusters in suitable models.
In \zcref{sec:main-thm}, we combine the methods of the two previous sections to prove the main theorem of this paper, completing Maksimova's classification.
 
    \section{Preliminaries}\label{sec:prelims}
    \paragraph{Basic notions.}
We assume that the reader is familiar with the elementary notions of relational semantics for normal unimodal logic, see e.g.\@ \cite{CZ1997,BdRV2001}.

A \emph{preorder} on a set \(\sX\) is a reflexive transitive relation \(\relPre\) on \(\sX\).
We write \(\vx \relPre[strict] \vy\) \tdefiff{} \(\vx \relPre \vy\) and \(\vy \relPre[not] \vx\).
A non-empty subset \(\sC \sSubsetEq \sX\) is a \emph{\(\relPre\)-cluster} or just \emph{cluster} if for all \(\vx,\vz\in \sC\), we have \(\vx \relPre \vz\), and, moreover, \(\vx \relPre \vy \relPre \vz\) implies \(\vy \in \sC\).
Note that if \(\vx, \vx[']\) are in some cluster \(\sC\) and \(\vy, \vy[']\) in a cluster \(\sD\), then \(\vx \relPre \vy\) iff \(\vx['] \relPre \vy[']\), and similarly for \(\relPre[strict]\).
Thus, instead of \(\vx \relPre \vy\), we will also write \(\sC \relPre \vy\), \(\vx \relPre \sD\) or \(\sC \relPre \sD\), and similarly for \(\relPre[strict]\).
A cluster \(\sC\sSubsetEq \sX\) is called \emph{final} if there does not exist a cluster \(\sD\sSubsetEq \sX\) with \(\sC \relPre[strict] \sD\).

Formulas are built from atomic propositions (or \emph{atoms} for short) using the constant \(\ltrue\), unary connectives \(\lnot\), \(\lBox\), and the binary connective \(\land\).
The connectives \(\lfalse\), \(\lor\), \(\limplies\), \(\liff\), and \(\lDiamond\) are defined as usual.
For a formula \(\lchi\) in atoms \(\lap[0],\dots, \lap[\nn-1]\) and formulas \(\lphi[0],\dots,\lphi[\nn-1]\) we will write \(\lchi[subst={\lphi[0],\dots,\lphi[\nn-1]}]\) for the formula obtained by substituting each instance of \(\lap[\nii]\) by \(\lphi[\nii]\).
For a logic \(\logModalL\) and a formula \(\lphi\) we will also write \(\lProves[\logModalL] \lphi\) if \(\lphi \in \logModalL\).

We say that a set of formulas \(\lthDelta\) is \emph{subformula closed} \tdefiff{} whenever \(\lphi \in \lthDelta\) and \(\lpsi\) is a subformula of \(\lphi\), then \(\lpsi \in \lthDelta\).
It is \emph{negation closed} \tdefiff{} \(\lphi \in \lthDelta\) implies \(\lnot\lphi \in \lthDelta\), and \emph{box closed} \tdefiff{} \(\lphi \in \lthDelta\) implies \(\lBox\lphi \in \lthDelta\).
The respective closures are defined in the standard way.

Recall that the logic \(\logKReflTrans\) is complete with respect to preorder models, i.e.\@ models \(\frmM = \structuple{\sX, \relPre, \fVal}\), where \(\structuple{\sX,\relPre}\) is a preordered set.
The logic \(\logKReflTransConfl\) is \(\logKReflTrans \oplus \paren{\lDiamond\lBox \lap \limplies \lBox\lDiamond\lap}\) and is complete with respect to preorder models \(\frmM = \structuple{\sX, \relPre, \fVal}\) that are \emph{confluent}, i.e.\@
for all \(\vx,\vy,\vz \in \sX\), if \(\vx \relPre \vy\) and \(\vx \relPre \vz\), then there exists \(\vw\in \sX\) with \(\vy \relPre \vw\) and \(\vz \relPre \vw\).

Let \(\frmM = \structuple{\frF, \fVal}\) be a model.
For a formula \(\lphi\), we write \(\sem[\frmM]{\lphi} \coloneqq \Set{\vx \in \sX}[\frmM,\vx \lValidates \lphi]\).
A \emph{definable variant} of \(\frmM\) is a model \(\frmM[']\) on \(\frF\) such that \(\sem{\lap}[\frmM[']]\) is \emph{definable} in \(\frmM\) for every atom \(\lap\), i.e.\@ there exists a formula \(\lphi\) such that \(\sem{\lap}[\frmM[']] = \sem{\lphi}[\frmM]\).
Finally, for \(\vx \in \sX\), we denote by \(\frmM[gen subfr=\vx]\) the  \emph{submodel generated by \(\vx\)} and its underlying frame by \(\frF[gen subfr=\vx]\).

For frames \(\frF = \structuple{\sX, \relPre}\) and \(\frG = \structuple{\sY, \relPre[']}\), a function \(\Function{\ff}{\sX}{\sY}\) is called a \emph{\tpmorph{} from \(\frF\) to \(\frG\)} \tdefiff{} it is surjective, order-preserving and whenever \(\ff{\vx} \relPre['] \vy[']\) then there exists \(\vx['] \in \sX\) such that \(\vx \relPre \vx[']\) and \(\ff{\vx[']} = \vy[']\).
For models \(\frmM\) and \(\frmN\) based on \(\frF\) and \(\frG\) respectively, \(\ff\) is a \emph{\tpmorph} (of models) from \(\frmM\) to \(\frmN\) \tdefiff{} additionally \(\sem{\lap}[\frmM] = \ff[preimage]{\sem{\lap}[\frmN]}\) for every atom \(\lap\).
In this case, \(\frG\) and \(\frmN\) are called \emph{p\babelhyphen{nobreak}morphic images} of \(\frF\) and \(\frmM\) respectively.
Recall that \tpmorph{}s of models preserve satisfaction of formulas, i.e.\@ \(\sem{\lphi}[\frmM] = \ff[preimage]{\sem{\lphi}[\frmN]}\) for any formula \(\lphi\).

\paragraph{Frame formulas.}
The logics above \(\logKReflTrans\) relevant for this paper can be axiomatized using Fine's \cite[Section~2]{Fine1974-ascending-chain-S4-logics} \emph{frame formulas}.
For \(\nn \in \Naturals\) we write \(\sNumber{\nn} \coloneqq \Set{0,\dots,\nn-1}\).

\begin{definition}[Frame formula]\zlabel{def:frame-formula}
    Let \(\nn \in \Naturals\) and \(\frG = \structuple{\sNumber{\nn}, \relPre}\) be a finite rooted preorder such that \(0\) is a root of \(\frG\).
    Define \(\lFrame{\frG}\) to be the negation of the conjunction of
\begin{tasks}[style = thmenumerate, label-align = right](2)
        \task \(\lap[0]\),
        \task \(\lBox{\lap[0] \lor \dots \lor \lap[\nn - 1]}\),
        \task \(\lBox{\lap[\nii] \limplies \lnot\lap[\nj]}\) for \(\nii \neq \nj\),
        \task \(\lBox{\lap[\nii] \limplies \lDiamond\lap[\nj]}\) for \(\nii \relPre \nj\), and
        \task \(\lBox{\lap[\nii] \limplies \lnot\lDiamond\lap[\nj]}\) for \(\nii \relPre[not] \nj\).
    \end{tasks}
\end{definition}

Note that we take the negation of the frame formulas defined by \textcite{Fine1974-ascending-chain-S4-logics},  as is standard convention in the literature \autocite{Zakharyaschev1992-canonical-formulas-for-K4-part-1,Bezhanishvili2006-lattices-of-intermediate-and-cylindric-modal-logics} and matches the characteristic formulas of \textcite{Jankov1968-strongly-independent-superintuitionistic-calculi}.

The frame formula of \(\frG\) is refuted on a frame \(\frF\) iff \(\frG\) is the p\babelhyphen{nobreak}morphic image of a generated subframe of \(\frF\) \autocite[Lemma~2.1]{Fine1974-ascending-chain-S4-logics}.
The proof, however, shows something stronger: it characterizes when a particular \emph{substitution of \(\lFrame{\frG}\)} is refuted in a point in a \emph{model}.
This will be crucial in our proof of the \tCIP{}.

\begin{lemma}\zlabel{lemma:p-morphism-fine-formula}
    Let \(\frmM\) be a model on a preorder \(\frF = \structuple{\sX, \relPre[\frF]}\), and \(\frG = \structuple{\sNumber{\nn}, \relPre[\frG]}\) be a finite rooted preorder.
    Then, for formulas \(\lphi[0], \dots,\lphi[\nn - 1]\) and \(\vx \in \sem{\lphi[0]}[\frmM]\), we have \(\vx \notin \sem{\lFrame{\frG}[subst={\lphi[0], \dots, \lphi[\nn - 1]}]}[\frmM]\) iff there exists a \tpmorph{} \(\ff\) from \(\frF[gen subfr=\vx]\) to \(\frG\) such that \(\ff[preimage]{\nii} = \sem{\lphi[\nii]}[\frmM[{gen subfr=\vx}]]\).
\end{lemma}

We include a proof for the sake of completeness.

\begin{proof}
    Define a model \(\frmN\) on \(\frG\) with atoms \(\lap[0] , \dots, \lap[\nn - 1]\) by defining \(\sem{\lap[\nii]}[\frmN] \coloneqq \Set{\nii}\), and define a definable variant \(\frmM[']\) of \(\frmM[gen subfr=\vx]\), again with atoms \(\lap[0] , \dots, \lap[\nn - 1]\), by \(\sem{\lap[\nii]}[\frmM[']] \coloneqq \sem{\lphi[\nii]}[\frmM[{gen subfr=\vx}]]\).
    Note that \(\sem{ \lFrame{\frG}[subst={\lphi[0], \dots, \lphi[\nn - 1]}] }[\frmM[{gen subfr=\vx}]] = \sem{ \lFrame{\frG} }[\frmM[']]\), and that a \tpmorph{} as in the lemma statement is exactly a \tpmorph{} of models from \(\frmM[']\) to \(\frmN\).
    Therefore, it suffices to prove that \(\vx \notin \sem{ \lFrame{\frG} }[\frmM[']]\) iff there exists a \tpmorph{} \(\ff\) from \(\frmM[']\) to \(\frmN\).

    (\(\Rightarrow\))
    Suppose \(\vx \notin \sem{ \lFrame{\frG} }[\frmM[']]\), so \(\vx\) satisfies the formulas listed in \zcref{def:frame-formula}.
    First note that for each \(\vy \in \sX\) with \(\vx \relPre \vy\), by formulas (ii) and (iii), there exists a unique \(\nii \in \sNumber{\nn}\) such that \(\vy \in \sem{\lap[\nii]}[\frmM[']]\).
    Define \(\ff{\vy}\) to be this \(\nii\).
    We show that \(\ff\) is a \tpmorph{}.

    Let \(\vy \in \sX\) and \(\nj \in \sNumber{\nn}\).
    Then \(\vy \in \sem{\lap[\ff{\vy}]}[\frmM[']]\).
    By formulas (iv) and (v), \(\ff{\vy} \relPre \nj\) iff \(\vy \in \sem{\lDiamond\lap[\nj]}[\frmM[']]\), giving the back- and forth-conditions.
    For surjectivity, let \(\nii \in \sNumber{\nn}\).
    Since \(0\) is a root of \(\frF\), we have \(0 \relPre \nii\).
    Moreover, \(\vx \in \sem{\lphi[0]}[\frmM]\), so \(\vx \in \sem{\lap[0]}[\frmM[']]\), yielding \(\ff{\vx} = 0\).
    Hence, by the back-condition, the \(\ff\)-preimage of \(\nii\) is non-empty.

    (\(\Leftarrow\))
    It is easy to check that the point \(0\) in \(\frmN\) satisfies the formulas listed in \zcref{def:frame-formula}, so it refutes the negation of the conjunction.
    Therefore, \(0 \notin \sem{\lFrame{\frG}}[\frmN]\).
    Now if \(\ff\) is a \tpmorph{} from \(\frmM[']\) to \(\frmN\), then it preserves satisfaction, so \(\ff[preimage]{0} \sIntersection \sem{ \lFrame{\frG} }[\frmM[']] = \sEmpty\).
    Since \(\vx \in \sem{\lphi[0]}[\frmM]\) and hence \(\vx \in \sem{\lap[0]}[\frmM[']]\), we get \(\vx \in \ff[preimage]{0}\) and conclude that \(\vx \notin \sem{ \lFrame{\frG} }[\frmM[']]\).
\end{proof}

For \(\nn \geq 1\), we define
\begin{itemize}
    \item \(\frCluster[\nn] = \structuple{\sNumber{\nn}, \relPre}\), where \(\nii \relPre \nj\) for any \(\nii, \nj \in \sNumber{\nn}\), and
    \item \(\frCluster[withtop,\nn] = \structuple{\sNumber{\nn + 1}, \relPre}\), where \(\nii \relPre \nj\) iff either \(\nii, \nj \in \sNumber{\nn}\) or \(\nj = \nn\).
\end{itemize}
Moreover, we define, for \(\nn \in \Naturals\),
\[
    \lCluFr[\nn] \coloneqq \lFrame{\frCluster[\nn+1]} , \quad
    \lCluFr[withtop,\nn] \coloneqq \lFrame{\frCluster[withtop,\nn+1]}, \mand
    \lCluFr[\omega] = \lCluFr[withtop,\omega] \coloneqq \ltrue .
\]

For a superintuitionistic logic \(\logLambda\) and \(\nm,\nn \leq \omega\), \textcite{Maksimova1980-interpolation-theorems} defines the logic \(\Gamma(\logLambda, \nm, \nn) \coloneqq \modcomp{\logLambda} \oplus \lCluFr[\nm] \oplus \lCluFr[withtop,\nn]\), where \(\modcomp{\logLambda}\) is the smallest modal companion of \(\logLambda\); see also \cite[73]{GM2005}.
As noted in the introduction, \(\Gamma(\logLambda,\nm,\nn)\) can be equivalently described as the logic complete with respect to the finite \tframes{\modcomp{\logLambda}} such that the size of final clusters is bounded by \(\nm\) and the size of non-final clusters is bounded by \(\nn\) \cite[Theorem 3.24, \citepage{74}]{GM2005}.
We are concerned with \(\logLambda \in \Set{\logSiIPC, \logSiKC}\), i.e.\@ \(\modcomp{\logSiIPC} = \logKReflTrans\) and \(\modcomp{\logSiKC} = \logKReflTransConfl\).

\begin{figure}
    \centering

    \begin{tikzpicture}[
    fcirc/.style={circle,draw=black, minimum size = 0.41em, inner sep = 0em}
    ]

    \node (north) at (0.4,0.01) {};
    \node (east) at (1.11,0) {};
    \node (south) at (0.4,-0.01) {};
    \node (west) at (-0.01,0) {};

    \node[draw, rounded corners = 0.75em,inner sep=0.42em,fit= (north) (south) (east) (west)] {};

    \node[fcirc] (0) at (0,0) {};
    \node[fcirc] (1) at (0.4,0) {};
    \node[fcirc] (2) at (1.1,0) {};
    \node[] (dots) at (0.75,0) {\({\dots}\)};
    \end{tikzpicture}
    \qquad
    \begin{tikzpicture}[
    fcirc/.style={circle,draw=black, minimum size = 0.41em, inner sep = 0em}
    ]

    \node (north) at (0.4,0.01) {};
    \node (east) at (1.11,0) {};
    \node (south) at (0.4,-0.01) {};
    \node (west) at (-0.01,0) {};

    \node[draw, rounded corners = 0.75em,inner sep=0.42em,fit= (north) (south) (east) (west)] {};

    \node[fcirc] (0) at (0,0) {};
    \node[fcirc] (1) at (0.4,0) {};
    \node[fcirc] (2) at (1.1,0) {};
    \node[fcirc] (3) at (0.55,1) {};
    \node[] (dots) at (0.75,0) {\({\dots}\)};

    \draw[->] (0.55,0.3) -- (0.55,0.9);
    \end{tikzpicture}
    \caption{The frames \(\frCluster[\nn]\) and \(\frCluster[\nn,withtop]\)}
    \label{fig:cluster-frames}
\end{figure}
 
    \section{Smoryński Models}\label{sec:smorynski}
    Our strategy for proving the \tCIP{} is based on Smoryński's \cite{Smorynski1978-Beths-theorem-and-self-referential-sentences} approach.
In this section we present the first step of Smoryński's method: given an implication \(\lphi[0] \limplies \lphi[1]\) which lacks an interpolant in an extension \(\logModalL\) of \(\logKReflTrans\), we construct a finite preorder model refuting \(\lphi[0] \limplies \lphi[1]\), the \emph{Smoryński model}.
This step is the standard part of Smoryński's method \cite{Boolos1980-systems-of-modal-logic-with-provability-interpretations,Maksimova1987-interpolation-in-normal-modal-logics,GM2005}, but in the literature it is often intertwined with the second, logic specific, part of the method.

We fix a normal extension \(\logModalL\) of \(\logKReflTrans\) for which we want to prove the \tCIP{}.
Additionally, we fix two  negation- and subformula-closed sets of formulas \(\lthSigma[\nii]\) for \(\nii \in \sTwo\) such that \(\lthSigma[\nii]\) is finite up to logical equivalence in \(\logKReflTrans\).
In practice, we will choose these sets to be certain closures of \(\lphi[0]\) and \(\lphi[1]\), respectively.
We write \(\lthSigma \coloneq \lthSigma[0] \sUnion \lthSigma[1]\).

The construction of the Smoryński model is similar to that of the canonical model, in that the points of the model will be sets of formulas.
There are, however, two differences.
First, and crucially, the consistency condition on the sets of formulas is replaced with a stronger \emph{inseparability} condition.
Second, to obtain a finite model, only subsets of \(\lthSigma\) are considered, which is finite up to logical equivalence.

For a subset \(\lthDelta \sSubsetEq \lthSigma\), we will write \(\lthDelta[\nii] \coloneqq \lthDelta \sIntersection \lthSigma[\nii]\).
Moreover, we write \(\lAND \lthDelta\) for some conjunction of formulas from \(\lthDelta\) such that every \(\lphi \in \lthDelta\) is logically equivalent in \(\logKReflTrans\) to one of the conjuncts.
Since \(\lthSigma\) is finite up to logical equivalence in \(\logKReflTrans\), \(\lAND \lthDelta\) always exists and, up to logical equivalence in \(\logKReflTrans\), does not depend on the choice of the conjuncts.
Finally, define \(\lthDelta[unbox] \coloneqq \Set{ \lphi }[ \lBox\lphi \in \lthDelta ]\) and \(\lBox\lthDelta \coloneqq \Set{ \lBox\lphi }[ \lphi \in \lthDelta ]\).

\begin{definition}
    Let \(\lthDelta \sSubsetEq \lthSigma\).
    We call \(\lthDelta\)
    \begin{itemize}
        \item \emph{separable} \tdefiff{} there exists a formula \(\lpsi\) with atoms in \(\lthSigma[0] \sIntersection \lthSigma[1]\) such that \(\lProves[\logModalL] \lAND \lthDelta[0] \limplies \lpsi\) and \(\lProves[\logModalL] \lAND \lthDelta[1] \limplies \lnot\lpsi\),
        \item \emph{inseparable} \tdefiff{} \(\lthDelta\) is not separable, and
        \item \emph{maximal inseparable} \tdefiff{} \(\lthDelta\) is inseparable and for each \(\lphi \in \lthSigma\), either \(\lphi \in \lthDelta\) or \(\lnot\lphi \in \lthDelta\).
\end{itemize}
\end{definition}

Note that for separability, we can assume without loss of generality that \(\lthDelta[0]\) and \(\lthDelta[1]\) are disjoint, in the following sense:

\begin{lemma}\zlabel{rmk:disjointness-for-separability}
	If \(\lthDelta\) is separable, then there exists a formula \(\lpsi\) with atoms in \(\lthSigma[0] \sIntersection \lthSigma[1]\) such that \(\lProves[\logModalL] \lAND \lthDelta[0] \limplies \lpsi\) and \(\lProves[\logModalL] \lAND \paren{\lthDelta[1] \sMinus \lthDelta[0]} \limplies \lnot\lpsi\).
\end{lemma}
\begin{proof}
    By assumption there is \(\lphi\) with atoms in \(\lthSigma[0] \sIntersection \lthSigma[1]\) such that \(\lProves[\logModalL] \lAND \lthDelta[0] \limplies \lphi\) and \(\lProves[\logModalL] \lAND \lthDelta[1] \limplies \lnot\lphi\).
    Define \(\lthDelta['] \coloneqq \lthDelta[0] \sIntersection \lthDelta[1]\).
    Then, by classical reasoning, it follows that \(\lProves[\logModalL] \lAND \lthDelta[0] \limplies \lnot\paren{\lAND\lthDelta['] \limplies \lnot\lphi}\) and \(\lProves[\logModalL] \lAND \paren{\lthDelta[1] \sMinus \lthDelta[']} \limplies \paren{\lAND\lthDelta['] \limplies \lnot\lphi}\).
    Therefore, \(\lpsi \coloneqq \lnot\paren{\lAND\lthDelta['] \limplies \lnot\lphi}\) is the desired formula.
\end{proof}

\begin{lemma}[{\cite[Lemma~5.16]{GM2005}}]\zlabel{lemma:maximal-separable-pairs}
\hfill
    \begin{thmenumerate}
        \item\zlabel{lemma:maximal-separable-pairs:extension}
        Every inseparable subset of \(\lthSigma\) can be extended to a maximal inseparable subset.
        \item\zlabel{lemma:maximal-separable-pairs:equivalence}
        For all maximal inseparable \(\lthDelta \sSubsetEq \lthSigma\), \(\nii \in \sTwo\), and \(\lphi \in \lthSigma[\nii]\),
        \[
            \lProves[\logModalL] \lAND \lthDelta[\nii] \limplies \lphi \iff \lphi \in \lthDelta[\nii].
        \]
        In particular, \(\lthDelta\) is consistent and for all \(\lphi,\lpsi \in \lthSigma[\nii]\), if \( \lProves[\logModalL] \lphi \liff \lpsi\) then
        \[
            \lphi \in \lthDelta[\nii] \iff \lpsi \in \lthDelta[\nii].
        \]
    \end{thmenumerate}
\end{lemma}

From these maximal inseparable sets we construct a finite model, which we will call the \emph{Smoryński model}, after Craig Smoryński~\cite{Smorynski1978-Beths-theorem-and-self-referential-sentences}.

\begin{construction}[Smoryński model]
    We define the \emph{Smoryński model} \(\frmM[\lthSigma] = \structuple{\sX[\lthSigma], \relPreEquiv, \fVal}\) as follows.
    Take for \(\sX[\lthSigma]\) the set of the maximal inseparable subsets of \(\lthSigma\), define
    \(\relPreEquiv\) by \(\lthDelta \relPreEquiv \lthDelta[']\) iff \(\lthDelta[unbox] \sSubsetEq \lthDelta[', unbox]\), and \(\fVal\) by \(
        \fVal{\lap} \coloneqq \Set{ \lthDelta \in \sX[\lthSigma] }[\lap \in \lthDelta]
    \) for every atom \(\lap \in \lthSigma\).
    Note that, by \zsubref{lemma:maximal-separable-pairs}{lemma:maximal-separable-pairs:equivalence}, \(\frmM[\lthSigma]\) is finite, since \(\lthSigma\) contains only finitely many formulas up to logical equivalence.
\end{construction}

Smoryński models satisfy the following Truth Lemma for formulas in \(\lthSigma\).

\begin{lemma}[Truth Lemma]\zlabel{lemma:truth-lemma}
    For every formula \(\lphi \in \lthSigma\) and \(\lthDelta \in \sX[\lthSigma]\),
    \[
        \frmM[\lthSigma], \lthDelta \lValidates \lphi \iff \lphi \in \lthDelta.
    \]
\end{lemma}
\begin{proof}
    We proceed by induction on the formula \(\lphi\).
    The case for atoms holds by definition and the steps for the Boolean connectives are routine.

    So assume \(\lphi = \lbox\lpsi\).
    Suppose that \(\lbox\lpsi \in \lthDelta\) and let \(\lthDelta['] \in \sX\) with \(\lthDelta \relPreEquiv \lthDelta[']\).
    Then \(\lbox\lpsi \in \lthDelta[']\) and, since \(\lProves[\logKReflTrans] \lbox \lpsi \limplies \lpsi\), it follows from \zsubref{lemma:maximal-separable-pairs}{lemma:maximal-separable-pairs:equivalence} that \(\lpsi \in \lthDelta[']\).
    Hence, by the induction hypothesis, \(\frmM[\lthSigma], \lthDelta['] \lValidates\lpsi\).
    We conclude that \(\frmM[\lthSigma], \lthDelta \lValidates\lbox\lpsi\).

    Conversely, suppose that \(\lbox\lpsi \notin \lthDelta\), so by maximality \(\lnot\lbox\lpsi \in \lthDelta\).
    Let us write \(\lthPsi \coloneqq \lBox\lthDelta[unbox]\).
    It suffices to prove that the set \(\lthPsi \sUnion \Set{\lnot\lpsi}\) is inseparable.
    Suppose that this is the case.
    Then, by \zsubref{lemma:maximal-separable-pairs}{lemma:maximal-separable-pairs:extension}, it extends to a maximal inseparable \(\lthDelta['] \sSubsetEq \lthSigma\), and clearly \(\lthDelta \relPreEquiv \lthDelta[']\).
    By the induction hypothesis, \(\frmM[\lthSigma], \lthDelta['] \lValidates \lnot\lpsi\), and we conclude \(\frmM[\lthSigma], \lthDelta \lValidates \lnot\lbox\lpsi\).

So assume for a contradiction that \(\lthPsi \sUnion \Set{\lnot\lpsi}\) is separable.
    Recall that \(\lnot\lbox\lpsi \in \lthDelta\).
    Without loss of generality, we assume that \(\lnot\lbox\lpsi \in \lthSigma[0]\), and hence \(\lnot\lpsi \in \lthSigma[0]\).
    By \zcref{rmk:disjointness-for-separability}, there exists a formula \(\ldelta\) with atoms in \(\lthSigma[0] \sIntersection \lthSigma[1]\) such that
    \begin{equation*}
        \lProves[\logModalL] \paren{ \lnot\lpsi \land \lAND\lthPsi[0] } \limplies \ldelta
        \quad \textrm{(a)} \mand
        \lProves[\logModalL] \lAND\lthPsi[1] \limplies \lnot\ldelta
        \quad \textrm{(b)} .
    \end{equation*}
    Applying, first, necessitation and the \tKaxiom{} to (b), and, second, the equivalence of \(\lBox[pow=2]\) and \(\lBox\) in \(\logKReflTrans\), yields
    \[
        \lProves[\logModalL] \paren{ \lAND\lBox\lthPsi[1] } \limplies \lbox\lnot\ldelta
        \mand
        \lProves[\logModalL] \paren{ \lAND\lthPsi[1] } \limplies \lbox\lnot\ldelta ,
    \]
    respectively.
    Applying classical reasoning to (a), then the same reasoning as for (b), and then classical reasoning again, we obtain
    \begin{gather*}
        \lProves[\logModalL] \paren{ \lnot\ldelta \land \lAND\lthPsi[0] } \limplies \lpsi, \qquad
        \lProves[\logModalL] \paren{ \lbox\lnot\ldelta \land \lAND\lthPsi[0] } \limplies \lbox\lpsi, \quad\text{and}\\
        \lProves[\logModalL] \paren{ \lnot\lbox\lpsi \land \lAND\lthPsi[0] } \limplies \lnot\lbox\lnot\ldelta .
    \end{gather*}
    However, since \(\lnot\lbox\lpsi \in \lthDelta[0]\) and \(\lthPsi \sSubsetEq \lthDelta\), this gives a separation of \(\lthDelta\), a contradiction.
\end{proof}

For logics with the confluence axiom, we can  make the Smoryński model confluent.

\begin{lemma}[cf.\@ {\cite[\citepage{146}]{GM2005}}]\zlabel{lemma:Smorinski-confluence}
    Suppose \(\lthSigma[0]\) and \(\lthSigma[1]\) are box closed, and \(\logModalL\) is a normal extension of \(\logKReflTransConfl\).
    Then the underlying frame of \(\frmM[\lthSigma]\) is confluent.
\end{lemma}
\begin{proof}
    Let \(\lthDelta[upper=0], \lthDelta[upper=1], \lthDelta[upper=2] \in \sX[\lthSigma]\) with \(\lthDelta[upper=0] \relPreEquiv \lthDelta[upper=1]\) and \(\lthDelta[upper=0] \relPreEquiv \lthDelta[upper=2]\).
    Write \(\lthPhi[upper=\nii] \coloneqq \lthDelta[upper=\nii, unbox]\).
    It suffices to show that \(\lBox\lthPhi[upper=1] \sUnion \lBox\lthPhi[upper=2]\) is inseparable, since then, by \zcref{lemma:maximal-separable-pairs}, it extends to a maximal inseparable \(\lthDelta[upper=3] \sSubsetEq \lthSigma\) with \(\lthDelta[upper=1] \relPreEquiv \lthDelta[upper=3]\) and \(\lthDelta[upper=2] \relPreEquiv \lthDelta[upper=3]\).

    So assume for a contradiction that  \(\lBox\lthPhi[upper=1] \sUnion \lBox\lthPhi[upper=2]\) is separable.
    Then there exists a formula \(\lpsi\) with atoms in \(\lthSigma[0] \sIntersection \lthSigma[1]\) such that
    \begin{equation*}
        \lProves[\logModalL] \lAND\paren{\lBox\lthPhi[upper=1, 0] \sUnion \lBox\lthPhi[upper=2, 0]} \limplies \lpsi \mand
        \lProves[\logModalL] \lAND\paren{\lBox\lthPhi[upper=1, 1] \sUnion \lBox\lthPhi[upper=2, 1]} \limplies \lnot\lpsi .
    \end{equation*}
    Thus, by applying necessitation, the \tKaxiom{}, and the equivalence of \(\lBox[pow=2]\) and \(\lBox\) in \(\logKReflTrans\),
    \begin{equation}\label{eq:lemma:Smorinski-confluence:boxed-separation}
        \lProves[\logModalL] \lAND\paren{\lBox\lthPhi[upper=1, 0] \sUnion \lBox\lthPhi[upper=2, 0]}  \limplies \lbox\lpsi \mand
        \lProves[\logModalL] \lAND\paren{\lBox\lthPhi[upper=1, 1] \sUnion \lBox\lthPhi[upper=2, 1]}  \limplies \lbox\lnot\lpsi .
    \end{equation}
    Now let \(\lphi \in \lthPhi[upper=1, \nii] \sUnion \lthPhi[upper=2, \nii]\).
    Since \(\lthSigma[\nii]\) is box and negation closed, \(\lDiamond\lBox\lphi \in \lthSigma[\nii]\).
    Moreover, from \(\lthDelta[upper=0] \relPreEquiv \lthDelta[upper=1]\) and \(\lthDelta[upper=0] \relPreEquiv \lthDelta[upper=2]\), we obtain \(\lDiamond\lBox\lphi \in \lthDelta[upper=0, \nii]\).
    Therefore,
    \[
        \lProves[\logModalL] \lAND\lthDelta[upper=0, \nii] \limplies \lAND\lDiamond\lBox{\lthPhi[upper=1, \nii] \sUnion \lthPhi[upper=2, \nii]} ,
    \]
    and since \(\lProves[\logKReflTransConfl] (\ldiamond\lbox\lap \land \ldiamond\lbox \laq) \limplies \ldiamond\lbox(\lap \land \laq)\),
    \[
        \lProves[\logModalL] \lAND\lthDelta[upper=0, \nii] \limplies \lDiamond\lBox\lAND\paren{\lthPhi[upper=1, \nii] \sUnion \lthPhi[upper=2, \nii]} .
    \]
    Combining this with \zcref[noname]{eq:lemma:Smorinski-confluence:boxed-separation}, we obtain
    \begin{equation*}
        \lProves[\logModalL] \lAND\lthDelta[upper=0, 0]  \limplies \lDiamond\lBox\lpsi \mand
        \lProves[\logModalL] \lAND\lthDelta[upper=0, 1] \limplies \lDiamond\lBox\lnot\lpsi .
    \end{equation*}
    Finally, since \(\lProves[\logKReflTransConfl] \ldiamond \lbox \lnot \lpsi \limplies \lnot \ldiamond\lbox \lpsi\), it follows that \(\lProves[\logModalL] \lAND\lthDelta[upper=0, 1] \limplies \lnot\ldiamond\lbox\lpsi\), thus separating \(\lthDelta[upper=0]\), a contradiction.
\end{proof}

Putting everything together, we obtain a model for refuting implications which lack interpolants.
For a set of formulas \(\lthDelta\), we call \(\lphi \in \lthDelta\) \emph{atomically saturated} (in \(\lthDelta\)) \tdefiff{} every atom that occurs in \(\lthDelta\) also occurs in \(\lphi\).

\begin{proposition}\zlabel{prop:Smorinski}
	Let \(\logModalL\) be a normal extension of \(\logKReflTrans\) and \(\lthSigma[0], \lthSigma[1]\) be sets of formulas that are finite up to logical equivalence in \(\logKReflTrans\).
	Then there exists a finite preorder model \(\frmM\) such that
	\begin{thmenumerate}
		\item for all atomically saturated \(\lphi \in \lthSigma[0]\) and \(\lpsi \in \lthSigma[1]\), if \(\lphi \limplies \lpsi\) lacks an interpolant in \(\logModalL\), then \(\frmM\) refutes \(\lphi \limplies \lpsi\),
		\item \(\frmM \lValidates \lphi\) for every \(\lphi \in \logModalL \sIntersection \lthSigma\), and
		\item if \(\logModalL\) extends \(\logKReflTransConfl\), then \(\frmM\) is confluent.
	\end{thmenumerate}
\end{proposition}
\begin{proof}
	For \(\nii \in \sTwo\), let \(\lthSigma[', \nii]\) be the (simultaneous) closure of \(\lthSigma[\nii]\) under boxes, negations, and subformulas.
	There are only finitely many non-equivalent modalities built from \(\lnot\) and \(\lbox\) in \(\logKReflTrans\) \autocite{Kuratowski1922-sur-l-operation-A-de-l-ananlysis-situs,Parry1939-modalities-in-the-survey-system-of-strict-implication}, so \(\lthSigma[', \nii]\) is finite up to logical equivalence in \(\logKReflTrans\).
	Let \(\frmM\) be the Smoryński model \(\frmM[\lthSigma[']]\) with parameters \(\logModalL\), \(\lthSigma[', 0]\) and \(\lthSigma[', 1]\).
	(i)\sentencespace
	Suppose that \(\lphi \in \lthSigma[0]\), \(\lpsi \in \lthSigma[1]\) are atomically saturated and \(\lphi \limplies \lpsi\) lacks an interpolant in \(\logModalL\).
	Then \(\Set{\lphi, \lnot\lpsi} \sSubsetEq \lthSigma[']\) is inseparable, so by \zsubref{lemma:maximal-separable-pairs}{lemma:maximal-separable-pairs:extension}, we can extend it to a maximal inseperable \(\lthDelta\), and by the Truth Lemma \(\frmM, \lthDelta \lValidates \lphi\) and \(\frmM, \lthDelta \lNotValidates \lpsi\).
	(ii)\sentencespace
	Let \(\lphi \in \logModalL \sIntersection \lthSigma\) and \(\lthDelta \in \sX[\lthSigma[']]\).
	By \zsubref{lemma:maximal-separable-pairs}{lemma:maximal-separable-pairs:equivalence}, \(\lphi \in \lthDelta\), so by the Truth Lemma \(\frmM, \lthDelta \lValidates \lphi\).
	(iii)\sentencespace
	Since \(\lthSigma[']\) is closed under boxes, the claim follows from \zcref{lemma:Smorinski-confluence}.
\end{proof}

For \(\Gamma(\logSiIPC, \omega, \omega) = \logKReflTrans\) and \(\Gamma(\logSiKC, \omega, \omega) = \logKReflTransConfl\), the \tCIP{} now follows immediately: if \(\lphi \limplies \lpsi\) lacks an interpolant then, by the proposition (taking \(\lthSigma[0] \coloneqq \Set{\lphi}\) and \(\lthSigma[1] \coloneqq \Set{\lpsi}\)), \(\lphi \limplies \lpsi\) is refuted in a (finite) model of the logic.
For other logics of the form \(\Gamma(\logLambda, \nm, \nn)\), however, the Smoryński model might not yet be a model of the logic, since in general it could contain clusters that are too large.
In the next section, we show how to overcome this problem.

    \section{Refining Clusters}\label{sec:refine-clusters}
    In this section we describe a transformation of frames that reduces the size of clusters, and establish a condition under which this does not change the satisfaction of formulas in a subformula-closed set \(\lthSigma\).
We will apply this construction to a suitably chosen Smoryński model in the next section to obtain a model of our given logic.

The transformation works by removing certain \enquote{arrows} from the accessibility relation.
For a cluster \(\sC\) and a given subset \(\fF{\sC} \sSubsetEq \sC\), we remove all (non-reflexive) arrows from points in \(\sC\) to points in \(\sC \sMinus \fF{\sC}\), i.e.\@ the cluster  \(\sC\)  is replaced by the  cluster \(\fF{\sC}\)  together with a number of incomparable points below it.
The transformation is depicted in \zcref{fig:cluster-refinement} and formalized as the following construction.

\begin{figure}
    \centering
    \begin{tikzpicture}[
    fcirc/.style={circle,draw=black, minimum size = 4.5pt, inner sep = 0pt},
    square/.style={regular polygon,regular polygon sides=4},
    fsqu/.style={square,draw=black, minimum size = 5.5pt, inner sep = 0pt},
    fheart/.style={shape=heart,scale=0.23,draw=black, minimum size = 4.5pt, inner sep = 8pt}
    ]

\node (north) at (0.4,0.01) {};
    \node (east) at (1.91,0) {};
    \node (south) at (0.4,-0.01) {};
    \node (west) at (-0.01,0) {};

    \node[draw, rounded corners = 8pt,inner sep=1.6mm,fit= (north) (south) (east) (west)] {};

    \node[fheart] (0) at (0,0) {};
    \node[fheart] (1) at (0.4,0) {};
    \node[fcirc] (2) at (0.8,0) {};
    \node[fcirc] (3) at (1.2,0) {};
    \node[fcirc] (n) at (1.9,0) {};
    \node[] (dots) at (1.55,0) {\({\dots}\)};
    \draw[->] (0.95,0.3) -- (0.95,0.75);

    \draw[rounded corners=8pt]
     (-1,1.55) -- (-0.5,0.8) -- (2.5,0.8) -- (3,1.55);

\draw[dashed,->] (3.4,0.5) to[out=45,in=135] (5.6,0.5);

\node (north2) at (0.95+7,0.01) {};
    \node (east2) at (1.16+7,0) {};
    \node (south2) at (0.95+7,-0.01) {};
    \node (west2) at (0.74+7,0) {};

    \node[draw, rounded corners = 8pt,inner sep=1.6mm,fit= (north2) (south2) (east2) (west2)] {};

    \node[fheart] (02) at (0.75+7,0) {};
    \node[fheart] (12) at (1.15+7,0) {};
    \node[fcirc] (22) at (0.4+7,-0.9) {};
    \node[fcirc] (32) at (0.95+7,-0.9) {};
    \node[fcirc] (n2) at (1.65+7,-0.9) {};
    \node[] (dots) at (1.3+7,-0.9) {\({\dots}\)};
    \draw[->] (0.95+7,0.3) -- (0.95+7,0.75);
    \draw[->] (22) -- (0.8+7,-0.325);
    \draw[->] (32) -- (0.95+7,-0.325);
    \draw[->] (n2) -- (1.1+7,-0.325);

    \draw[rounded corners=8pt]
     (-1+7,1.55) -- (-0.5+7,0.8) -- (2.5+7,0.8) -- (3+7,1.55);
    \end{tikzpicture}
    \caption{Visualization of the transformation of \zcref{constr:refining-clusters} applied to a single cluster. The hearts represent the elements of \(\fF{\sC}\), while the circles represent elements of \(\sC \sMinus \fF{\sC}\).}\zlabel{fig:cluster-refinement}
\end{figure}

\begin{construction}\zlabel{constr:refining-clusters}
    Let \(\frmM = \structuple{\sX, \relPre, \fVal}\) be a preorder model and \(\fF\) a function that maps every cluster \(\sC\) of \(\frmM\) to a non-empty subset \(\fF{\sC} \sSubsetEq \sC\).
	We define the relation \(\relPre[\fF]\) on \(\sX\) by \(\vx \relPre[\fF] \vy\) iff
    \[
        \vx \relPre \vy \text{ and if } \vx, \vy \in \sC \text{ for some cluster } \sC \text{ then either } \vx = \vy \text{ or } \vy \in \fF{\sC} ,
    \]
	and the model \(\frmM[\relF] \coloneqq \structuple{\sX, \relPre[\relF], \fVal}\).
\end{construction}

The following lemma follows immediately from the definition.

\begin{lemma}\zlabel{lemma:refining-clusters:frame-properties}
	In \zcref{constr:refining-clusters}, the following hold:
	\begin{thmenumerate}
	\item The model \(\frmM[\relF]\) is a preorder model which is confluent if \(\frmM\) is confluent.
	\item The function \(\relF\) monotonically maps clusters of \(\frmM\) to clusters of \(\frmM[\relF]\) such that every final cluster and every proper cluster of \(\frmM[\relF]\) is in the image of \(\relF\). In particular, it maps final clusters to final clusters and non-final clusters to non-final clusters.
	\end{thmenumerate}
\end{lemma}

For the transformation to preserve satisfaction of formulas in a subformula-closed set \(\lthSigma\), we need the following property for each \(\fF{\sC}\).

\begin{definition}[Adequate subset]
     Let \(\lthSigma\) be a set of formulas, \(\frmM = \structuple{\sX, \relPre, \fVal}\) a preorder model, and \(\sC \subseteq \sX\) a cluster in \(\frmM\).
     A subset \(\sD \sSubsetEq \sC\) is called \emph{\tadequate{\lthSigma}} \tdefiff{} for all \(\lbox\lphi \in \lthSigma\), if \(\sD \sSubsetEq \sem{\lphi \land \lnot\lBox\lphi}\) then there exists \(\vz \in \sem{\lnot\lphi}\) such that \(\sC \relPre[strict] \vz\).
\end{definition}

\begin{lemma}\zlabel{lemma:refining-clusters:satisfaction}
	Let \(\lthSigma\) be a subformula-closed set of formulas.
	If, in \zcref{constr:refining-clusters}, \(\fF{\sC} \sSubsetEq \sC\) is \tadequate{\lthSigma} for every cluster \(\sC\), then \(\sem{\lphi}[\frmM] = \sem{\lphi}[\frmM[\relF]]\) for each \(\lphi \in \lthSigma\).
\end{lemma}
\begin{proof}
    By induction on the formulas \(\lphi \in \lthSigma\).
    The case for atoms is immediate, and the steps for the Boolean connectives are routine.
    So assume as induction hypothesis \(\sem{\lphi}[\frmM] = \sem{\lphi}[\frmM[\relF]]\).
    We prove that \(\sem{\lBox\lphi}[\frmM] = \sem{\lBox\lphi}[\frmM[\relF]]\).
    We will omit the subscript from \(\sem{\lphi}\) since it is irrelevant by the induction hypothesis.

    Let \(\vx \in \sem{\lBox\lphi}[\frmM]\).
    Then for all \(\vy \in \sX\) with \(\vx \relPre \vy\), we have \(\vy \in \sem{\lphi}\).
    Now, since \({\relPre[\relF]} \sSubsetEq {\relPre}\), it follows that also for all \(\vy['] \in \sX\) with \(\vx \relPre[\relF] \vy[']\), we have \(\vy['] \in \sem{\lphi}\).
    Hence, \(\vx \in \sem{\lBox\lphi}[\frmM[\relF]]\).

    Conversely, let \(\vx \in \sX \sMinus \sem{\lBox\lphi}[\frmM]\).
    Let \(\sC\) be the cluster of \(\vx\).
    Since points in a cluster satisfy the same boxed formulas, \(\sC \sSubsetEq \sem{\lnot\lBox\lphi}[\frmM]\).
    Suppose there exists \(\vy \in \fF{\sC} \sIntersection \sem{\lnot\lphi}\).
    Then \(\vx \relPre[\relF] \vy\) (since \(\vy \in \fF{\sC}\)), and hence \(\vx \notin \sem{\lBox\lphi}[\frmM[\relF]]\), and we are done.

    So suppose \(\fF{\sC} \sIntersection \sem{\lnot\lphi} = \sEmpty\), i.e.\@ \(\fF{\sC} \sSubsetEq \sem{\lphi}\).
    Also, \(\fF{\sC} \sSubsetEq \sC \sSubsetEq \sem{\lnot\lBox\lphi}\), and thus \(\fF{\sC} \sSubsetEq \sem{\lphi \land \lnot\lBox\lphi}\).
    Now, since \(\fF{\sC}\) is \tadequate{\lthSigma} in \(\frmM\), there exists \(\vz \in \sem{\lnot\lphi}\) such that \(\sC \relPre[strict] \vz\).
    In particular, \(\vx \relPre[\relF] \vz\), so \(\vx \notin \sem{\lBox\lphi}[\frmM]\).
\end{proof}

Thus, we have reduced the problem of decreasing cluster sizes to finding small adequate subsets of clusters.
We will prove the existence of such a subset if the points of the cluster satisfy \enquote{enough} substitutions of the frame formulas \(\lCluFr[\nn] = \lFrame{\frCluster[\nn+1]}\) and \(\lCluFr[withtop,\nn] = \lFrame{\frCluster[withtop,\nn+1]}\).

 \begin{definition}[Relative satisfaction]
     Let \(\lthSigma\) be a set of formulas, \(\frmM\) a preorder model, and \(\lchi\) a formula in atoms \(\lap[0], \dots, \lap[\nn - 1]\).
     We say that a cluster \(\sC\) in \(\frmM\) \emph{satisfies \(\lchi\) relative to \(\lthSigma\)} if for all \(\lphi[0], \dots, \lphi[\nn - 1] \in \lthSigma\), we have \(\sC \sSubsetEq \sem{\lchi[{subst={\lphi[0], \dots, \lphi[\nn - 1]}}]}\).
 \end{definition}

 \begin{lemma}\zlabel{lemma:adequate-manytons}
     Let \(\lthSigma\) be a Boolean and subformula closed set of formulas, \(\frmM = \structuple{\sX, \relPre, \fVal}\) a preorder model, \(\sD\) a \tadequate{\lthSigma} subset of a cluster \(\sC\) in \(\frmM\), \(\vx \in \sD\), and \(\nn \in \Naturals[strict pos]\).
     If
     \begin{itemize}
         \item \(\sC\) is final and satisfies \(\lCluFr[\nn]\) relative to \(\lthSigma\), or
         \item \(\sC\) is non-final and satisfies \(\lCluFr[\nn, withtop]\) relative to \(\lthSigma\),
     \end{itemize}
     then \(\sD\) contains a \tadequate{\lthSigma} subset of cardinality at most \(\nn\) extending \(\Set{\vx}\).
 \end{lemma}
\begin{proof}
    By contraposition.
    Suppose that \(\sD\) does not contain a \tadequate{\lthSigma} subset of cardinality at most \(\nn\) extending \(\Set{\vx}\) and let \(\sU\) be the strict upset generated by \(\vx\), i.e.\@
    \(
        \sU \coloneqq \Set{\vz \in \sX}[\vx \relPre[strict] \vz]
    \).
    We define a sequence of distinct points \(\vy[0], \dots, \vy[\nn] \in \sD\) by induction as follows
    and write \(\sD[\nk] \coloneqq \Set{\vy[0], \dots, \vy[\nk - 1]}\).

    Define \(\vy[0] \coloneqq \vx\).
    Now suppose \(\vy[\nk]\) is already defined and \(\nk < \nn\).
    Then \(\sD[\nk + 1]\) has cardinality \(\nk + 1 \leq \nn\), hence \(\sD[\nk + 1]\) is not \tadequate{\lthSigma}.
    Thus, there exists \(\lBox\lphi[\nk] \in \lthSigma\) such that \(\sD[\nk + 1] \sSubsetEq \sem{\lphi[\nk] \land \lnot\lBox\lphi[\nk]}\) and \(\sU \sSubsetEq \sem{\lphi[\nk]}\).
    Since \(\sD\) is \tadequate{\lthSigma}, there exists \(\vy[\nk + 1] \in \sD\) such that \(\vy[\nk + 1] \notin \sem{\lphi[\nk] \land \lnot\lBox\lphi[\nk]}\).
    Because \(\vx\) and \(\vy[\nk + 1]\) are in the same cluster, they satisfy the same boxed formulas, so \(\vy[\nk + 1] \in \sem{\lnot\lBox\lphi[\nk]}\).
    We conclude that \(\vy[\nk + 1] \notin \sem{\lphi[\nk]} \sSupersetEq \sD[\nk + 1]\), yielding \(\vy[\nk + 1]\notin \sD[\nk + 1]\).

    Define \(\lpsi[\sU] \coloneqq \lBox\lphi[0]\), and notice that \(\sU \sSubsetEq \sem{\lphi[0]}\). Therefore, since \(\sU\) is upward closed, \(\sU \sSubsetEq \sem{\lpsi[\sU]}\).
    Moreover, \(\vx \notin \sem{\lpsi[\sU]}\), and therefore, since \(\lpsi[\sU]\) is a boxed formula, \(\sC \sIntersection \sem{\lpsi[\sU]} = \sEmpty\).
    For each \(\nk \in \sNumber{\nn + 1}\), define \[
        \lpsi[\nk] \coloneqq \lnot\lphi[\nk - 1] \land \lphi[\nk] \land \dots \land \lphi[\nn - 1] \land \lnot\lpsi[\sU]
    \] (for the \(\nk = 0\) case take \(\lphi[-1] \coloneqq \lfalse\) so that the first conjunct holds vacuously).
    Notice that \(\vy[\nk] \in \sem{\lpsi[\nk]}\), and the \(\sem{\lpsi[\nk]}\) are pairwise disjoint and their union contains \(\sC\).
    Finally, note that, since each \(\lBox\lphi[\nl] \in \lthSigma\) and \(\lthSigma\) is subformula and Boolean closed, \(\lpsi[\nk] \in \lthSigma\).

    Now consider the submodel \(\frmM[gen subfr={\vx}]\) generated by \(\vx\), and let \(\frF[gen subfr={\vx}]\) be the underlying frame.
    Note that the domain of \(\frF[gen subfr={\vx}]\) is precisely \(\sC \sUnion \sU\), and restricted to this subframe \(\sem{\lpsi[\sU]}\) coincides with \(\sU\) and the \(\sem{\lpsi[\nk]}\) form a partition of \(\sC\).
    Define a function \(\Function{\ff}{\sC \sUnion \sU}{\sNumber{\nn + 2}}\) by
    \begin{equation*}
        \ff{\vz} \coloneqq \begin{cases}
            \nk & \text{if } \vz \in \sem{\lphi[\nk]} \text{ for some } \nk \in \sNumber{\nn + 1}, \\
            \nn + 1 & \text{if } \vz \in \sem{\lphi[\sU]}.
        \end{cases}
    \end{equation*}

    First, suppose \(\sC\) is final, so \(\sU = \sEmpty\).
    Then this function is surjective onto \(\sNumber{\nn + 1}\), and easily seen to be a \tpmorph{} from \(\frF[gen subfr={\vx}]\) to \(\frCluster[\nn + 1]\).
    By \zcref{lemma:p-morphism-fine-formula}, \(\vx\) refutes
    \(
        \lCluFr[\nn, subst={\lpsi[0], \dots, \lpsi[\nn]}]
    \),
    so \(\vx\) does not satisfy \(\lCluFr[\nn]\) relative to \(\lthSigma\).

    Second, suppose \(\sC\) is non-final.
    Then \(\sU \neq \sEmpty\) so this function is surjective, and easily seen to be a \tpmorph{} from \(\frF[gen subfr={\vx}]\) to \(\frCluster[\nn + 1, withtop]\).
    By \zcref{lemma:p-morphism-fine-formula}, \(\vx\) refutes
    \(
        \lCluFr[\nn, withtop, subst={\lpsi[0], \dots, \lpsi[\nn], \lpsi[\sU]}]
    \),
    so \(\vx\) does not satisfy \(\lCluFr[\nn, withtop]\) relative to \(\lthSigma\).
\end{proof}

Notice that satisfying \(\lchi\) relative to \(\lthSigma[0] \sUnion \lthSigma[1]\) is strictly stronger than satisfying \(\lchi\) relative to \(\lthSigma[0]\) and relative to \(\lthSigma[1]\), since the former allows substituting \(\lchi\) with a mix of formulas from \(\lthSigma[0]\) and \(\lthSigma[1]\).
In the Smoryński model, we can only hope for the clusters to satisfy \(\lCluFr[\nn]\) or \(\lCluFr[withtop, \nn]\) relative to each of \(\lthSigma[0]\) and \(\lthSigma[1]\), but not their union, since the Truth Lemma does not apply to formulas that mix atoms from both sets.
However, to refine the clusters of the model, we will need a method to obtain (small enough) subsets which are adequate for both sets.
When \(\nn \in \Set{1, 2}\), the following variation on \zcref{lemma:adequate-manytons} provides exactly that.

\begin{lemma}[Main Lemma]\zlabel{lemma:bi-adequate-doubletons}
    Let \(\frmM = \structuple{\sX, \relPre, \fVal}\) be a preorder model, \(\nn \in \Set{1, 2}\) and \(\lthSigma[\nii]\) Boolean and subformula closed sets of formulas, for \(\nii \in \sTwo\).
    Let \(\sC\) be a final cluster in \(\frmM\) which satisfies \(\lCluFr[\nn]\) relative to \(\lthSigma[0]\) and relative to \(\lthSigma[1]\), or a non-final cluster which satisfies \(\lCluFr[\nn, withtop]\) relative to \(\lthSigma[0]\) and \(\lthSigma[1]\).
    Then there exists a \(\lthSigma[0]\)- and \tadequate{\lthSigma[1]} subset of \(\sC\) with cardinality at most \(\nn\).
\end{lemma}

\begin{proof}
    Let \(\vx \in \sC\).
	First, suppose \(\nn = 1\).
	By \zcref{lemma:adequate-manytons} applied with \(\lthSigma \coloneqq \lthSigma[0]\) and \(\sD \coloneqq \sC\), there exists a \tadequate{\lthSigma[0]} singleton in \(\sC\) extending \(\Set{\vx}\).
	Therefore, \(\Set{\vx}\) is \tadequate{\lthSigma[0]}.
	It follows similarly that \(\Set{\vx}\) is \tadequate{\lthSigma[1]}.

	For \(\nn = 2\), we need a small combinatorial argument.
	By \zcref{lemma:adequate-manytons}, there exist a \tadequate{\lthSigma[0]} subset \(\Set{\vx, \vy} \sSubsetEq \sC\) and a \tadequate{\lthSigma[1]} subset \(\Set{\vx, \vz} \sSubsetEq \sC\).
	Now \(\sD \coloneqq \Set{\vx, \vy, \vz}\) is \(\lthSigma[0]\)- and \tadequate{\lthSigma[1]}.
	By \zcref{lemma:adequate-manytons}, with \(\lthSigma \coloneqq \lthSigma[1]\), the set \(\sD \sSubsetEq \sC\), and base-point \(\vy\), there exists a \tadequate{\lthSigma[1]} singleton or doubleton contained in \(\sD\) extending \(\Set{\vy}\).
	If it is \(\Set{\vx, \vy}\) or \(\Set{\vy}\), then \(\Set{\vx, \vy}\) is \(\lthSigma[0]\)- and \tadequate{\lthSigma[1]}, and we are done.
	So suppose not.
	Then \(\Set{\vy, \vz}\) is \tadequate{\lthSigma[1]}.
	Again, by  \zcref{lemma:adequate-manytons}, now with \(\lthSigma \coloneqq \lthSigma[0]\) and base-point \(\vz\), we find a \tadequate{\lthSigma[0]} singleton or doubleton contained in \(\sD\) extending \(\Set{\vz}\).
	If it is \(\Set{\vx, \vz}\) or \(\Set{\vz}\), then \(\Set{\vx, \vz}\) is \(\lthSigma[0]\)- and \tadequate{\lthSigma[1]}, and we are done;
otherwise \(\Set{\vy, \vz}\) is \tadequate{\lthSigma[0]}, yielding that \(\Set{\vy, \vz}\) is \(\lthSigma[0]\)- and \tadequate{\lthSigma[1]}.
\end{proof}

\begin{remark}
	Note that the Main Lemma fails to hold for \(\nn \in \Naturals\) with \(\nn \geq 3\), because the combinatorics fail at that point.
	This is expected, since \(\Gamma(\logLambda, \nm, \nn)\) fails to have the \tCIP{} when \(\nm\) or \(\nn\) is finite and at least 3.
\end{remark}
 
    \section{Main Theorem}\label{sec:main-thm}
    Now we are ready to show the main result of this paper, but before we proceed with the proof, we fix the following useful notation.
For a set of formulas \(\lthSigma\), we denote by \(\boolsub{\lthSigma}\) the Boolean closure of the subformula closure of \(\lthSigma\).
Additionally, for a formula \(\lchi\) in atoms \(\lap[0], \dots, \lap[\nn-1]\), we define the \emph{\(\lthSigma\)-substitution of \(\lchi\)} by
\[
    \subst{\lchi}{\lthSigma} \coloneqq \Set{\lchi[{subst={\lphi[0], \dots, \lphi[\nn - 1]}}]}[\lphi[0], \dots, \lphi[\nn - 1] \in \lthSigma] .
\]

\begin{theorem}[Main Theorem]\zlabel{thm:mainthm}
    Let \(\logLambda \in \Set{\logSiIPC, \logSiKC}\) and \(\nm, \nn \in \Set{1, 2, \omega}\).
    Then the logic \(\logGamma = \Gamma(\logLambda, \nm, \nn)\) has the Craig interpolation property.
\end{theorem}

\begin{proof}
    Recall that \(\logGamma = \modcomp{\logLambda} \oplus \lCluFr[\nm] \oplus \lCluFr[\nn,withtop]\).

    Let \(\lphi[0], \lphi[1]\) be formulas such that \(\lphi[0] \limplies \lphi[1]\) lacks an interpolant in \(\logGamma\).
    Define \(\lthSigma[\nii,'] \coloneqq \boolsub{\lphi[\nii]}\) and \(\lthSigma[\nii] \coloneqq \lthSigma[\nii,'] \sUnion \subst{\lCluFr[\nm]}{\lthSigma[\nii,']} \sUnion \subst{\lCluFr[\nn,withtop]}{\lthSigma[\nii,']}\) for \(\nii \in \sTwo\).
    Note that \(\lthSigma[\nii]\) is finite up to logical equivalence (even in \(\logK\)).
    Moreover, \(\lphi[\nii]\) is atomically saturated in \(\lthSigma[\nii]\).
    By \zcref{prop:Smorinski}, there exists a finite preorder model \(\frmM\) such that
\begin{center}
        (i) \(\frmM\) refutes \(\lphi[0] \limplies \lphi[1]\),\sentencespace
        (ii) \(\frmM \lValidates \logGamma \sIntersection \lthSigma\),\sentencespace and
        (iii) \(\frmM\) is confluent if \(\logLambda = \logSiKC\).
    \end{center}
    Note that by (iii), \(\frmM\) is based on a \tframe{\modcomp{\logLambda}}.

    To conclude the proof, we will refine the clusters in \(\frmM\) that are too large, to obtain a model refuting \(\lphi[0] \limplies \lphi[1]\) based on a \tframe{\logGamma}.
    Note that, since \(\lCluFr[\nm], \lCluFr[\nn,withtop] \in \logGamma\) and, by construction, \(\lthSigma[\nii]\) contains the \(\lthSigma[',\nii]\)-substitutions of \(\lCluFr[\nm]\) and \(\lCluFr[\nn,withtop]\), it follows from (ii) that every cluster in \(\frmM\) satisfies \(\lCluFr[\nm]\) and \(\lCluFr[\nn,withtop]\) relative to \(\lthSigma[',\nii]\) for \(\nii \in \sTwo\).

    We define a function \(\fF\) from the clusters of \(\frmM\) to subsets of these clusters, such that \(\fF{\sC}\) is a \(\lthSigma[', 0]\)- and \tadequate{\lthSigma[', 1]} subset of \(\sC\) for each \(\sC\).
    Let \(\sC\) be a final cluster.
    If \(\nm = \omega\), define \(\fF{\sC} \coloneqq \sC\).
    If \(\nm \in \Set{1, 2}\), since \(\sC\) satisfies \(\lCluFr[\nm]\) relative to \(\lthSigma[', 0]\) and \(\lthSigma[', 1]\), by \zcref{lemma:bi-adequate-doubletons}, there exists a \(\lthSigma[', 0]\)- and \tadequate{\lthSigma[', 1]} subset \(\sC['] \sSubsetEq \sC\) of size at most \(\nm\).
    Define \(\fF{\sC} \coloneqq \sC[']\).
    Next, let \(\sC\) be a non-final cluster.
    If \(\nn = \omega\), define \(\fF{\sC} \coloneqq \sC\).
    If \(\nn \in \Set{1, 2}\), since \(\sC\) satisfies \(\lCluFr[\nn, withtop]\) relative to \(\lthSigma[', 0]\) and \(\lthSigma[', 1]\), by \zcref{lemma:bi-adequate-doubletons}, there exists a \(\lthSigma[', 0]\)- and \tadequate{\lthSigma[', 1]} subset \(\sC['] \sSubsetEq \sC\) of size at most \(\nn\).
    Define \(\fF{\sC} \coloneqq \sC[']\).

    Now with \(\fF\) defined, construct a model \(\frmM[\fF]\) as in \zcref{constr:refining-clusters}.
By \zcref{lemma:refining-clusters:frame-properties}, the frame underlying \(\frmM[\fF]\) is a \tframe{\logGamma}.
    Moreover, since every \(\fF{\sC}\) is \(\lthSigma[', 0]\)- and \tadequate{\lthSigma[', 1]}, by \zcref{lemma:refining-clusters:frame-properties}, the transformation preserves the satisfaction of formulas in \(\lthSigma[', 0]\) and \(\lthSigma[', 1]\), hence also of formulas in \(\boolsub{\lthSigma[', 0] \sUnion \lthSigma[', 1]}\).
    In particular, \(\frmM[\relF]\) still refutes \(\lphi[0] \limplies \lphi[1]\).
\end{proof}

\zcref{thm:mainthm} completes Maksimova's classification of the normal extensions of \(\logKReflTrans\) with the Craig interpolation property, see \cite[Theorem~8.45, pp.~256\textendash{}257]{GM2005}.
In particular, this resolves Problem~14.3 of \textcite[p.~469]{CZ1997}.
Since the logics, for which the status of the interpolation properties remained open, have the Craig interpolation property and, hence, also the deductive interpolation property, this completes the classification result for the deductive interpolation property as well.
The superintuitionistic logics with the \tCIP{} are listed in \zcref{tab:int-extensions}, see \cite{Maksimova1978-Craigs-theorem}.

\begin{table}
    \caption{The superintuitionistic logics with the Craig interpolation property.}\zlabel{tab:int-extensions}
    \centering
    \begin{tblr}{colspec={ll}}
        \toprule
        Logic   & Axiomatization \\
        \midrule
        \(\logSiIPC\)&
        Intuitionistic logic
        \\
        \(\logSiKC\) &
        \(\logSiIPC + {\lnot\lap \lor \lnot\lnot\lap}\)
        \\
        \(\logSiLPtwo\) &
        \(\logSiIPC + {\lap \lor (\lap \limplies (\laq \lor \lnot\laq))}\)
        \\
        \(\logSiLV\) &
        \(\logSiLPtwo + {(\lap \limplies \laq) \lor (\laq \limplies \lap) \lor (\lap \liff \lnot\laq)}\)
        \\
        \(\logSiLS\) &
        \(\logSiLPtwo + {\lnot\lap \lor \lnot\lnot\lap}\)
        \\
        \(\logSiLC\) &
        \(\logSiIPC + {(\lap \limplies \laq) \lor (\laq \limplies \lap)}\)
        \\
        \(\logCPC\) &
        \(\logSiIPC + {\lap \lor \lnot\lap}\)
        \\
        \(\logForm\) & \(\logSiIPC + \lfalse\)
        \\
        \bottomrule
    \end{tblr}
\end{table}

\begin{corollary}\zlabel{cor:main-corollary}
    There are exactly 37 normal extensions of \(\logKReflTrans\) with the Craig interpolation property and exactly 49 with the deductive interpolation property.
    In particular:
    \begin{thmenumerate}
        \item\zlabel{cor:main-corollary:CIP} A normal extension of \(\logKReflTrans\) has the Craig interpolation property if and only if it is one of the following logics:
        \begin{align*}
            &\logForm && (\text{the inconsistent logic});
            \\
            &\Gamma(\logLambda, \nm, \nn),
            && \logLambda \in \Set{\logSiIPC, \logSiKC},\ \nm, \nn \in \Set{1, 2, \omega};
            \\
            &\Gamma(\logLambda, \nn, 1),\ \Gamma(\logLambda, 1, \nn),
            && \logLambda \in \Set{\logSiLPtwo,\logSiLV,\logSiLS},\ \nn \in \Set{1, 2, \omega};
            \\
            &\Gamma(\logCPC, \nn, 0),
            &&\nn \in \Set{1, 2, \omega}.
        \end{align*}
        \item\zlabel{cor:main-corollary:DIP} A normal extension of \(\logKReflTrans\) has the deductive interpolation property if and only if it is a logic mentioned in \textup{(i)} or it is one of the following logics:
        \begin{align*}
            &\Gamma(\logLambda, \nm, \nn), \quad \logLambda \in \Set{\logSiLPtwo, \logSiLV, \logSiLS},\ \nm, \nn \in \Set{2, \omega}.
        \end{align*}
    \end{thmenumerate}
\end{corollary}

    \paragraph{Declaration on AI\@.}
    The authors have not employed any generative AI tools in relation to this research and manuscript.

    \printbibliography

@BOOK{CZ1997,
  AUTHOR = {Chagrov, Alexander and Zakharyaschev, Michael},
  LOCATION = {Oxford},
  PUBLISHER = {Oxford University Press},
  DATE = {1997},
  GENDER = {pm},
  ISBN = {0 19 853779 4},
  NUMBER = {35},
  SERIES = {Oxford Logic Guides},
  TITLE = {Modal Logic},
}

@BOOK{BdRV2001,
  AUTHOR = {Blackburn, Patrick and de Rijke, Maarten and Venema, Yde},
  LOCATION = {Cambridge},
  PUBLISHER = {Cambridge University Press},
  DATE = {2001},
  GENDER = {pm},
  ISBN = {978 0 521 52714 9},
  NUMBER = {53},
  SERIES = {Cambridge Tracts in Theoretical Computer Science},
  TITLE = {Modal Logic},
}

@BOOK{GM2005,
  AUTHOR = {Gabbay, Dov M. and Maksimova, Larisa Lvovna},
  LOCATION = {Oxford},
  PUBLISHER = {Oxford University Press},
  DATE = {2005},
  GENDER = {pp},
  ISBN = {0 19 851174 4},
  NUMBER = {46},
  SERIES = {Oxford Logic Guides},
  TITLE = {Interpolation and Definability},
}

@ARTICLE{Fine1974-ascending-chain-S4-logics,
  AUTHOR = {Fine, Kit},
  DATE = {1974},
  DOI = {10.1111/j.1755-2567.1974.tb00081.x},
  JOURNALTITLE = {Theoria},
  NUMBER = {2},
  PAGES = {110--116},
  TITLE = {An ascending chain of {S4} logics},
  VOLUME = {40},
}

@ARTICLE{Jankov1968-strongly-independent-superintuitionistic-calculi,
  AUTHOR = {Jankov, Vadim Anatolyevich},
  TRANSLATOR = {Yablonski, A.},
  ORIGLANGUAGE = {russian},
  PUBLISHER = {American Mathematical Society},
  DATE = {1968},
  JOURNALTITLE = {Soviet Mathematics},
  NUMBER = {4},
  PAGES = {806--807},
  TITLE = {Constructing a sequence of strongly independent superintuitionistic propositional calculi},
  VOLUME = {9},
}

@ARTICLE{Kuratowski1922-sur-l-operation-A-de-l-ananlysis-situs,
  AUTHOR = {Kuratowski, Casimir},
  LANGUAGE = {french},
  DATE = {1922},
  DOI = {10.4064/fm-3-1-182-199},
  ISSN = {0016-2736},
  JOURNALTITLE = {Fundamenta Mathematicae},
  PAGES = {182--199},
  TITLE = {Sur l'opération Ā de l'Analysis Situs},
  VOLUME = {3},
}

@ARTICLE{Parry1939-modalities-in-the-survey-system-of-strict-implication,
  AUTHOR = {Parry, William Tuthill},
  PUBLISHER = {Association for Symbolic Logic},
  DATE = {1939-12},
  DOI = {10.2307/2268714},
  EPRINT = {2268714},
  EPRINTTYPE = {jstor},
  ISSN = {0022-4812},
  JOURNALTITLE = {The Journal of Symbolic Logic},
  NUMBER = {4},
  PAGES = {137--154},
  TITLE = {Modalities in the Survey System of Strict Implication},
  VOLUME = {4},
}

@ARTICLE{Maksimova1978-Craigs-theorem,
  AUTHOR = {Maksimova, Larisa Lvovna},
  ORIGLANGUAGE = {russian},
  PUBLISHER = {Plenum Publishing Corporation},
  DATE = {1978},
  DOI = {10.1007/BF01670006},
  ISSN = {0002-5232},
  JOURNALTITLE = {Algebra and Logic},
  ORIGDATE = {1977-11},
  PAGES = {427--455},
  TITLE = {{Craig's} theorem in superintuitionistic logics and amalgamable varieties of pseudo-{Boolean} algebras},
  VOLUME = {16},
}

@INPROCEEDINGS{Smorynski1978-Beths-theorem-and-self-referential-sentences,
  AUTHOR = {Smoryński, Craig},
  EDITOR = {Macintyre, Angus and Pacholski, Leszek and Paris, Jeff},
  PUBLISHER = {North-Holland Publishing Company},
  BOOKTITLE = {Logic Colloquium '77},
  DATE = {1978},
  DOI = {10.1016/S0049-237X(08)72008-1},
  EVENTDATE = {1977-08-01/1977-08-12},
  ISBN = {0 444 85178 X},
  ISSN = {0049-237X},
  NUMBER = {96},
  PAGES = {253--261},
  SERIES = {Studies in Logic and the Foundations of Mathematics},
  TITLE = {Beth's Theorem and Self-Referential Sentences},
  VENUE = {Wrołcaw, Poland},
}

@ARTICLE{Maksimova1980-interpolation-theorems,
  AUTHOR = {Maksimova, Larisa Lvovna},
  ORIGLANGUAGE = {russian},
  PUBLISHER = {Plenum Publishing Corporation},
  DATE = {1980},
  DOI = {10.1007/BF01673502},
  ISSN = {0002-5232},
  JOURNALTITLE = {Algebra and Logic},
  ORIGDATE = {1979-09},
  PAGES = {348--370},
  TITLE = {Interpolation theorems in modal logics and amalgamable varieties of topological {Boolean} algebras},
  VOLUME = {18},
}

@ARTICLE{Boolos1980-systems-of-modal-logic-with-provability-interpretations,
  AUTHOR = {Boolos, George},
  DATE = {1980-04},
  DOI = {10.1111/j.1755-2567.1980.tb00686.x},
  JOURNALTITLE = {Theoria},
  NUMBER = {1},
  PAGES = {7--18},
  TITLE = {On Systems of Modal Logic with Provability Interpretations},
  VOLUME = {46},
}

@ARTICLE{Maksimova1981-interpolation-theorems-sufficient-conditions,
  AUTHOR = {Maksimova, Larisa Lvovna},
  ORIGLANGUAGE = {russian},
  PUBLISHER = {Plenum Publishing Corporation},
  DATE = {1981},
  DOI = {10.1007/BF01669837},
  ISSN = {0002-5232},
  JOURNALTITLE = {Algebra and Logic},
  NUMBER = {2},
  ORIGDATE = {1980-03},
  PAGES = {120--132},
  TITLE = {Interpolation theorems in modal logics. {Sufficient} conditions},
  VOLUME = {19},
}

@ARTICLE{Rautenberg1983-modal-tableau-calculi-and-interpolation,
  AUTHOR = {Rautenberg, Wolfgang},
  DATE = {1983},
  DOI = {10.1007/BF00249258},
  ISSN = {0022-3611},
  JOURNALTITLE = {Journal of Philosophical Logic},
  PAGES = {403--423},
  TITLE = {Modal tableau calculi and interpolation},
  VOLUME = {12},
}

@ARTICLE{Maksimova1984-absence-CIP-Dummett,
  AUTHOR = {Maksimova, Larisa Lvovna},
  ORIGLANGUAGE = {russian},
  PUBLISHER = {Plenum Publishing Corporation},
  DATE = {1984},
  DOI = {10.1007/BF01982206},
  ISSN = {0002-5232},
  JOURNALTITLE = {Algebra and Logic},
  NUMBER = {6},
  ORIGDATE = {1982-11},
  PAGES = {460--463},
  TITLE = {Absence of the interpolation property in the consistent normal modal extensions of the {Dummett} logic},
  VOLUME = {21},
}

@ARTICLE{Maksimova1987-interpolation-in-normal-modal-logics,
  AUTHOR = {Maksimova, Larisa Lvovna},
  LANGUAGE = {russian},
  LOCATION = {Kishinev},
  PUBLISHER = {Izdatel’stvo Shtiintsa},
  URL = {https://gdz.sub.uni-goettingen.de/id/PPN516257471},
  DATE = {1987},
  ISSN = {0542-9994},
  JOURNALSUBTITLE = {Neklassicheskie logiki},
  JOURNALTITLE = {Matematicheskie Issledovaniya},
  PAGES = {40--56},
  TITLE = {Interpolation in normal modal logics},
  VOLUME = {98},
}

@ARTICLE{Maksimova1991-amalgamation-and-interpolation-in-normal-modal-logics,
  AUTHOR = {Maksimova, Larisa Lvovna},
  DATE = {1991-09},
  DOI = {10.1007/BF00370682},
  ISSN = {1572-8730},
  JOURNALTITLE = {Studia Logica},
  NUMBER = {3},
  PAGES = {457--471},
  TITLE = {Amalgamation and interpolation in normal modal logics},
  VOLUME = {50},
}

@ARTICLE{Shimura1992-cut-free-systems-for-some-modal-logics-containing-S4,
  AUTHOR = {Shimura, Tatsuya},
  LOCATION = {Kraków, Poland},
  PUBLISHER = {Jagiellonian University Press},
  DATE = {1992},
  ISSN = {0137-2904},
  JOURNALTITLE = {Reports on Mathematical Logic},
  PAGES = {39--65},
  TITLE = {Cut-Free Systems for some Modal Logics Containing {S4}},
  VOLUME = {26},
}

@ARTICLE{Zakharyaschev1992-canonical-formulas-for-K4-part-1,
  AUTHOR = {Zakharyaschev, Michael},
  PUBLISHER = {Association for Symbolic Logic},
  DATE = {1992-12},
  DOI = {10.2307/2275372},
  EPRINT = {2275372},
  EPRINTTYPE = {jstor},
  JOURNALTITLE = {The Journal of Symbolic Logic},
  NUMBER = {4},
  PAGES = {1377--1402},
  TITLE = {Canonical Formulas for {K4}. {Part} {I}: Basic Results},
  VOLUME = {57},
}

@THESIS{Bezhanishvili2006-lattices-of-intermediate-and-cylindric-modal-logics,
  AUTHOR = {Bezhanishvili, Nick},
  INSTITUTION = {Universiteit van Amsterdam},
  DATE = {2006-01-26},
  TITLE = {Lattices of Intermediate and Cylindric Modal Logics},
  TYPE = {phdthesis},
}

@BOOK{Handbook2026-Craig-interpolation,
  EDITOR = {ten Cate, Balder and Jung, Jean Christoph and Koopmann, Patrick and Wernhard, Christoph and Wolter, Frank},
  PUBLISHER = {Ubiquity Press},
  DOI = {10.5334/bdg},
  PUBSTATE = {inpress},
  TITLE = {Theory and Applications of {Craig} Interpolation},
}

@INBOOK{Handbook2026-Craig-interpolation-chapter-3,
  AUTHOR = {Fussner, Wesley},
  CHAPTER = {3},
  CROSSREF = {Handbook2026-Craig-interpolation},
  DATE = {2025-12-01},
  EPRINT = {2512.01600},
  EPRINTCLASS = {math.LO},
  EPRINTTYPE = {arxiv},
  TITLE = {Interpolation in Non-Classical Logics},
}
\end{document}